\documentclass[a4paper, 10pt]{amsart}
\usepackage{graphicx}
\usepackage{amsfonts}
\usepackage{indentfirst,bm,amsmath}
\usepackage{amsthm,amssymb}
\usepackage{mathrsfs}
\usepackage{hyperref}

\newtheorem{theorem}{Theorem}[section]
\newtheorem{corollary}[theorem]{Corollary}
\newtheorem{lemma}[theorem]{Lemma}
\newtheorem{proposition}[theorem]{Proposition}
\newtheorem{definition}[theorem]{Definition}
\newtheorem{remark}[theorem]{Remark}
\newtheorem{example}[theorem]{Example}

\begin{document}

\title[]{Riemann-Hurwitz theorem and second main theorem for harmonic morphisms on graphs or metrized complexes}

\author[T. B. Cao \and M. N. Cheng]{ Tingbin Cao \and Mengnan Cheng}

\address[]{Department of Mathematics, Nanchang University, Nanchang city, Jiangxi 330031, P. R. China}
\email{tbcao@ncu.edu.cn, chengmengnan@email.ncu.edu.cn}

\thanks{This paper was partially supported by the National Natural Science Foundation of China (\#11871260).}

\date{}



\begin{abstract}In this article, we mainly obtain the Riemann-Hurwitz theorems for harmonic morphisms on (vertex-weighted) metric graphs or metrized complexes of algebraic curves, inspired of the recent work on harmonic morphisms of graphs or metrized complexes due to many researchers. By making use of these Riemann-Hurwitz theorems, we then systematically establish the second main theorems for harmonic morphisms on finite graphs, vertex-weighted graphs, (vertex-weighted) metric graphs or metrized complexes of algebraic curves, from the viewpoint of Nevanlinna theory.
\end{abstract}
\maketitle
\tableofcontents

\section{Introduction}
It is known that both Riemann-Roch theorem and Riemann-Hurwitz formula are fundamental results in the theory of divisors on smooth projective curves \cite{E.Arbarello-1985, E.Arbarello-2011},  have been actively and deeply studied. Tropical geometry is a new branch of mathematics, which makes a deep connection between algebraic geometry and combinatorial objects, and provides another connection between graph theory and the theory of algebraic curves. The analogue of an algebraic curve in tropical geometry is an (abstract) tropical curves, which following Mikhalkin \cite{MZ-2008}, can be considered simply as a metric graph. With the development of graph theory, tropical and non-Archimedean geometry in recent years, it is very interesting  that many fundamental theorems in classical algebraic geometry have combinatorial correspondence in tropical geometry. In 2000, Urakawa \cite{Ura-2000} firstly proposed harmonic morphisms on finite simple graphs. In 2007, M. Baker and S. Norine \cite{BN-2007} firstly introduced divisors on finite loopless multigraphs, and obtained a Riemann-Roch theorem on finite graphs. In 2013, O. Amini and L. Caporaso \cite{AC-2013} extended this to the Riemann-Roch theorem on weighted graphs. Their results also been developed to (vertex-weighted) metric graphs and metrized complexes of curves (see \cite{AB-2015, AC-2013, GK-2008, MZ-2008} and therein references). In addition, in \cite{BN-2009} M. Baker and S. Norine studied the category of finite graphs with harmonic morphisms as a discrete analogue of the category of Riemann surfaces with holomorphic maps, and go on to derive a Riemann-Hurwitz formula on finite graphs. It is revealed by Baker \cite{Bak-2008} that the theory of divisors on graphs and tropical curves are not just a formal analogies on that of curves. This arises many researches  \cite{Cap-2014, Ana-2000, BJ-2016, Cha-2013, Mik-2006, BBM-2011, AC-2013, Bak-2008, GK-2008, HKN-2013, MZ-2008, ABBR1-2015, ABBR2-2015} to focus on this topic, and even to higher dimensions (see \cite{cartwright2020tropical, sumi2021tropical} and therein references).\par

The second main theorem in Nevanlinna theory can be regarded as transcendental version of the Riemann-Hurwitz theorem. In 1960, S. S. Chern \cite{chern1960complex} considered the Nevanlinna theory for holomorphic map on Riemann surfaces. The second main theorem for algebraic curves in Nevanlinna theory \cite{ru-ugur-2017, ru-2021} sates that for a nonconstant holomorphic mapping $f$ from a compact Riemann surface $S$ with genus $g$ into one-dimensional complex projective space $\mathbb{P}(\mathbb{C}),$ we have $(q-2)\deg(f)\leq |E|+2(g-1)$ for any $q$ distinct points $a_{1}, \ldots, a_{q}\in\mathbb{P}(\mathbb{C}),$ where $|E|$ is the cardinality of the set $E=f^{-1}\{a_{1}, \ldots, a_{q}\}.$ Thus it is natural and interesting to consider the second main theorem on graphs and  metrized complexes of algebraic curves from the point view of Nevanlinna theory.\par

Let $\mathfrak{C}$, $\mathfrak{C'}$ be metrized complexes of algebraic curves on an algebraically closed field $\kappa$ whose underlying vertex-weighted metric graphs are $(\Gamma, w)$ and $(\Gamma', w'),$ respectively. Let $\varphi=(\phi, \{\phi_v\}_{v\in G})$ be a harmonic morphism defined as in Definition \ref{D7.1}. By introducing the definitions of genus and canonical divisor on metrized complexes of algebraic curves with vertex-weighted metric graphs in Subsection 2.5, we prove the \emph{Riemann-Hurwitz theorem} (see Theorem \ref{the7.6}) \emph{for harmonic morphisms $\varphi$ between metrized complexes $\mathfrak{C}$ and $\mathfrak{C'}$ of algebraical curves over algebraically closed field $\kappa$ as
$$\mathcal{K}_{(\mathfrak{C}, w)}={\varphi}^*\mathcal{K}_{(\mathfrak{C'}, w')}+R_{\varphi},$$
where $\mathcal{K}_{(\mathfrak{C}, w)}$ and $\mathcal{K}_{(\mathfrak{C'}, w')}$ are canonical divisors on $\mathfrak{C}$ and $\mathfrak{C'}$ respectively, $R_{\varphi}$ is the ramification divisor\begin{eqnarray*}
R_{\varphi}=\sum_{v\in V(G)}\left(K_v+A_v-\varphi^{*}K_{v'}-\varphi^{*}A_{v^{'}}\right)+\sum_{v\in V(G)}2\left(w(v)-w'(v')M_{\phi}(v)\right)(v),
\end{eqnarray*} and $v'=\phi(v).$} By this Riemann-Hurwitz theorem, we will obtain the \emph{second main theorem} (Theorem \ref{T7.7}) that \emph{
\begin{eqnarray*}
&&(q+g(\mathfrak{C'}, w')-1)\deg(\varphi)\leq g(\mathfrak{C}, w)-1+|E\cap V(G)|-\sum_{v\in V(G)}\left(g_v-M_{\phi}(v)g_{v^{'}}\right)\\&&-\sum_{v\in V(G)}\left(w(v)-w'(v')M_{\phi}(v)\right)-\frac{1}{2}\sum_{v\in V(G)}(val(v)-M_{\phi}(v)val(v')),
\end{eqnarray*} holds for any $q$ distinct vertices $\{a_1, \ldots , a_q\}\subset V(G')$, where $g(\mathfrak{C}, w)$ and $g(\mathfrak{C'}, w')$ are genus of metrized complexes $\mathfrak{C}$ and $\mathfrak{C'}$ respectively,  $E={\phi}^{-1}(\{a_1, \ldots, a_q\}),$ and $|E\cap V(G)|$ is the cardinality of $E\cap V(G).$}\par
\vskip6pt

The remainder of this paper is organized as follows. In the second section, we will introduce the preliminaries on theory of divisors on finite graphs, vertex-weighted graphs, metric graphs, vertex-weighted metric graphs,  and metrized complexes of algebraic curves for underlying vertex-weighted metric graphs. In Section 3 and Section 4, inspired by the Riemann-Hurwitz theorems by M. Baker and S. Norine \cite{BN-2009} on finite graphs and by L. Caporaso \cite{Cap-2014} on vertex-weighted graphs, we get the second main theorems on finite graphs and vertex-weighted graphs, respectively. In Section 5 and Section 6, we propose the Riemann-Hurwitz theorems and second main theorems on metric graphs and vertex-weighted metric graphs, respectively. At the last section, by modifying the definitions of genus and canonical divisors in Subsection 2.5, we use the harmonic morphism on a metrized complex of algebraic curves from a vertex-weighted metric graph, and then establish the Riemann-Hurwitz theorem and second main theorem on metrized complexes of algebraic curves. Some examples are given to explain our second main theorems.\par

\section{Preliminaries on the theory of divisors}
In this section, we recall the theory of divisors on a finite graph, vertex-weighted graph, (vertex-weighted) metric graph, and (vertex-weighted) metrized complexes of algebraic curves.

\subsection{Theory of divisors on a finite graph}\par
In this subsection, for the reader's convenience, we firstly recall some basic terms from graph theory, and choose conventions that apply both to combinatorics and to algebraic geometry (for more details, refer to see \cite{AC-2013, BN-2007, Bak-2008, Cap-2014} and the references therein).\par

A multigraph is a graph which is permitted to have multiple edges, a graph with no multiple edges is called simple. Throughout this paper, a \emph{finite graph} denoted by $G$ means an unweighted, finite connected multigraph.  We will denote by $V(G)$ and $E(G)$, respectively, the set of vertices and edges of $G$. To every edge $e \in E(G)$ one associates the pair $\{v, v'\}$ of possibly equal vertices which form the boundary of $e,$ we call $v$ and $v'$ the endpoints of the edge $e.$ If $v=v',$ then we say that $e$ is a loop-edge based at $v.$\par

A leaf is a pair $(e, v)$ of a vertex and an edge, where $e$ is the unique edge adjacent to $v.$ We say the $e$ is a leaf-edge and $v$ is a leaf-vertex. A edge $e \in E(G)$ is called a bridge if $ G\setminus e$ is disconnected. In particular, one leaf-edge is a bridge. \par

For a vertex $v \in V(G)$ and an edge $e \in E(G)$, if $e$ is adjacent to $v$, then we write $v \in e$. \par

We call the \emph{ valency} $val(v)$ (or degree $\deg(v)$) of a vertex $v \in V(G)$ is the number of edges having $v$ as end point. We agree that a loop based at $v$ is counted twice.\par

The \emph{genus} $g$ of $G$ is its first Betti number $$g=b_1(G)=|E(G)|-|V(G)|+1.$$
This is the dimension of the cycle space of $G$, and the integer $g=g(G)$ is called the ``cyclomatic number" of $G$.\par

\begin{definition} Given a finite graph $G$, we write $Div(G)$ for the free Abelian group on $V(G)$. An element $D$ of $Div(G)$ is called a divisor on $G$, and is written as a sum $$D=\sum_{v \in V(G)}D(v)(v),$$ where $D(v)\in \mathbb{Z}.$ We say that $D$ is effective, and write $D\geq 0$, if  $D(v) \geq 0$ for all $v \in V(G)$. For $D\in Div(G)$, the degree of a divisor $D$ is defined by the formula $$\deg(D)=\sum_{v \in V(G)}D(v).$$ Denote by $$Div_+(G)=\{D\in Div(G): D\geq 0\}$$ the set of effective divisors on $G$, and by $Div^0(G)$ the set of divisors of degree zero on $G$.\end{definition}

\begin{definition} The canonical divisor $K_G\in Div(G)$ of a finite graph $G$ is defined as $$K_G:=\sum_{v \in V(G)}(val(v)-2)(v).$$ \end{definition}

According to the $val(v)$ of a loop based at $v$ is counted twice and the sum over all vertices $v$ of $val(v)$ equals twice the number of edges in $G$, it is easy to deduce directly that  $$\deg(K_G)=2|E(G)|-2|V(G)|=2g-2$$  (or by the Riemann-Roch theorem for finite graphs with loops \cite[Theorem 3.6]{AC-2013}).\par

Following \cite{AC-2013}, let $G$ be a graph and let $\{e_1, \ldots, e_c\}\subset E(G)$ be the set of its loop-edges. We denote by $\hat{G}$ the new graph obtained by inserting one vertex in the interior of the loop-edge $e_j,$ for all $j=1, \ldots, c$. Observe that $\hat{G}$ has no loops and has the same genus with $G.$ Let $U=\{u_1, \ldots, u_c\}\subset V(\hat{G})$ be the set of vertices added to $G,$ and thus $V(G)=V(\hat{G})\setminus U.$ Because the vertices in $U$ are all $2$-valent, it is clear that the valencies of $G$ and $\hat{G}$ are the same, and hence the canonical divisor $K_{\hat{G}}$ of $\hat{G}$ is
$$K_{\hat{G}}=\sum_{\hat{v} \in V(\hat{G})}(val(\hat{v})-2)(\hat{v})=\sum_{\hat{v} \in V(\hat{G})\setminus U}(val(\hat{v})-2)(\hat{v}).$$
On the basis of graph theory, the sum over all vertices $v$ of $val(v)$ equals twice the number of edges in $G$, so we have $$\deg(K_{\hat{G}})=2|E(\hat{G})|-2|V(\hat{G})|=2g-2$$  (or by the Riemann-Roch theorem for finite graphs with no loops \cite[Theorem 1.12]{BN-2007}).\par

Therefore, we now get that the canonical divisors of $G$ and $\hat{G}$ preserves the same degrees (valencies). Hence, for a graph $G$ with loops we can deal with it by adding a point in the middle of each loop, and turning it into $\hat{G}.$\par

\subsection{Theory of divisors on a vertex-weighted graph}
In this subsection, following \cite{AC-2013, SK-2015}, we briefly recall some properties of vertex-weighted graphs.\par

A \emph{vertex-weighted graph} is a pair $(G, w)$, by which we mean that $G$ is a finite graph and a function $w: V(G)\to \mathbb{Z}_{\geq 0}$ called a weight function on the vertices. The genus, $g(G,w)$, of $(G,w)$ is $$g(G,w)=b_1(G)+\sum_{v \in V(G)}w(v).$$ \par
 
Define a divisor $D$ on $(G, w)$ by $$D=\sum_{v \in V(G)}D(v)(v),$$ where $D(v)\in \mathbb{Z}.$ For any vertex-weighted graph $(G, w),$ its divisor group $Div(G, w)$ is defined as the free Abelian group generated by the vertices of $G$.\par

For a vertex-weighted graph $(G, w),$ we associate to it a \emph{weightless graph} $G^w,$ which is obtained by attaching at every vertex $v\in V(G),$ $w(v)$ loops (or ``$1$-cycles"), denoted by $C_{v}^{1}, \ldots, C_{v}^{w(v)}.$ The new finite graph $G^w$ is called the virtual (weightless) graph. The $C_v^i$ are virtual loops. Notice that the initial graph $G$ is a subgraph of $G^w,$  we have $V(G)=V(G^w)$ and $g(G,w)=g(G^w)$. For the group of divisors of the vertex-weighted graph $(G,w),$ we have $$Div(G,w)=Div(G^w)=Div(G).$$

\begin{definition}The canonical divisor of $(G, w)$ is defined as the canonical divisor of $G^w,$  namely, $$K_{(G,w)}:=K_{(G^w)}=\sum_{v \in V(G^w)}(val_{G^w}(v)-2)v=\sum_{v\in V(G)}(2w(v)-2+val(v))v.$$ \end{definition}

It is easy to deduce directly that $$\deg(K_{(G,w)})=2g(G^w)-2=2g(G,w)-2$$ (or by the Riemann-Roch theorem for vertex-weighted graphs with loops \cite[Theorem 3.8]{AC-2013}).\par

\subsection{Theory of divisors for metric graphs}\par

Let's start by recalling the theory of divisors on  metric graphs. We refer the reader to \cite{Bak-2008, Cha-2013, GK-2008, HKN-2013, MZ-2008, SK-2015} for more details and more references.\par

\begin{definition} A metric graph $\Gamma$ is a metric space such that there exists a finite graph $G$ and a length function $\ell: E(G)\to \mathbb{R}_{>0}$ so that $\Gamma$ is obtained from $(G, \ell)$ by gluing intervals $[0, \ell(e)]$ for $e\in E(G)$ at their endpoints, as prescribed by the combinatorial data of $G.$ The distance $d(x, y)$ between two points $x$ and $y$ in $\Gamma$ is given by the length of the shortest path between them. In this case, we say that $(G, \ell)$ is a model for $\Gamma$.\end{definition}

If $G$ has no loops, then $(G, \ell)$ is called a loopless model. It is possible that a given metric graph $\Gamma$ admits many models $(G, \ell).$ For example, a line segment of length $a$ can be subdivided into many edges whose lengths sum to $a.$ Hence almost all points in $\Gamma$ have valence $2.$\par

Suppose that $\Gamma$ is not a circle. Let $V\subset \Gamma$ be the set of all points of a metric graph $\Gamma$ of valence different from $2,$ where the valence is the number of connected components of $U_{x} \setminus \{x\}$ with $U_{x}$ being any sufficiently small connected neighborhood of $x$ in $\Gamma$. Then define a model $(G_{V}, \ell)$ as follows \cite{Cha-2013}: the vertices of the graph $G_{V}$ are the points in $V,$ and the edges of $G_{V}$ correspond to the connected components of $\Gamma\setminus V.$ These components are necessarily isometric to open intervals, the length of each of which determines the function $\ell: E(G_{V})\rightarrow\mathbb{R}_{>0}.$ Then $(G_{V}, \ell)$ is a model for $\Gamma.$ Usually, we call the pair $(G_{V}, \ell)$ \emph{a canonical module} for $\Gamma,$ and denote  by $(G_0, \ell).$ If the $(G_0, \ell)$ has loops, then replace an additional vertex at the midpoint of each loop edge. We denote the \emph{canonical loopless model }by $(G_{-}, \ell)$ obtained from $(G_{0}, \ell).$ \par

A divisor $D$ on a metric graph $\Gamma$ is an element of the free Abelian group $Div(\Gamma)$ generated by points of $\Gamma$ defined by $$D=\sum_{x\in {\Gamma}} D(x)(x),$$ where $D(x)\in \mathbb{Z}.$  The degree of $D$ is defined by $\deg(D)=\sum_{x\in \Gamma}D(x).$ If $D(x)\geq 0$ for any $x\in {\Gamma}$, then the divisor is effective. \par

\begin{definition} The canonical divisor of a metric graph $\Gamma$ is $$K_{\Gamma}:=\sum_{v \in {V(G)}}(val(v)-2)(v)=\sum_{x \in {\Gamma}}(val(x)-2)(x).$$\end{definition}

It follows from the Riemann-Roch theorem for metric graphs \cite[Proposition 3.1]{GK-2008} and by taking the special divisors zero and $K_{\Gamma}$ that $$\deg(K_{\Gamma})=2g(\Gamma)-2,$$ where the genus $g(\Gamma)$ of a metric graph $\Gamma$ is defined to be its first Betti number, which equals $g(G)$ of any model $(G, \ell)$ of $\Gamma.$\par

\subsection{Theory of divisors for vertex-weighted metric graphs}\par
In this subsection, we introduce the divisors on vertex-weighted metric graphs \cite{AC-2013, BJ-2016, SK-2015}.\par

\begin{definition}
A vertex-weighted metric graph $(\Gamma, w)=(G, w, \ell),$ that is , $\Gamma$ is a metric graph with a model $(G, \ell)$ and a weighted function $w:\Gamma\to {\mathbb{Z}}_{\geq 0}$ such that $w(v)=0$ for all but finitely many point $v$ in $\Gamma.$ \end{definition}

We also denote $(\Gamma, \underline{0})=(G, \underline{0}, \ell)$ to be a pure metric graph $\Gamma$ with a model  $(G, \ell).$

\begin{definition} A pseudo-metric graph is a pair $(G, \ell)$ where $G$ is a finite graph and $\ell$ a pseudo-length function $\ell: E(G)\rightarrow\mathbb{R}_{\geq 0}$ which is allowed to vanish only on loop-edges of $G$ (that is, if $\ell(e)=0$ then $e$ is a loop-edge of $G$).
\end{definition}

Associate to a vertex-weighted metric graph $(\Gamma, w),$ the \emph{pseudo-metric graph} $(G^w, \ell^w)$ is defined as follows: $G^w$ is obtained by attaching to $G$ exactly $w(v)$ loops based at every vertex $v\in  V(G),$ and the pseudo-length function $\ell^w: E(G^w)\to \mathbb{R}_{\geq 0}$ is the extension of $\ell$ vanishing at all the virtual loops. Clearly, the pair $(G^w, \ell^w)$ is uniquely determined.\par

Conversely, to any pseudo-metric graph $(G^{'}, \ell^{'})$ we can associate a unique vertex-weighted metric graph $(G, w, \ell)$ satisfying $G^{'}=G^{w}$ and $\ell^{'}=\ell^{w}$ as follows. $G$ is the subgraph of $G^{'}$ obtained by removing every loop-edge $e\in E(G)$ such that $\ell^{'}(e)=0.$ Next, the length function $\ell$ is the restriction of $\ell^{'}$ to $G;$ finally, for any $v\in V(G)=V(G^{'})$ the weight $w(v)$ is defined to be the number of loop-edges of $G^{'}$ adjacent to $v$ and having length zero.\par

Notice that the pseudo-metric graph $(G^{'}, \ell^{'})$ associate to  a vertex-weighted metric graph $(\Gamma, w)$ is not a metric graph.  Amini and Caporaso \cite{AC-2013} defined the \emph{pure metric graph} $\Gamma_{\varepsilon}^w,$ for every $\varepsilon>0,$ $$\Gamma_{\varepsilon}^w=(G^w, \underline{0},  \ell_{\varepsilon}^w),$$ where $\ell_{\varepsilon}^w(e)=\varepsilon$ for every edge lying in some virtual cycle, and $\ell_{\varepsilon}^w(e)=\ell(e)$ otherwise. Hence $(G^w, \ell^w)=\lim_{\varepsilon\rightarrow 0}\Gamma_{\varepsilon}^w.$ \par

\begin{definition}The genus of a vertex-weighted metric graph $(\Gamma,w)$ is defined by $$g(\Gamma,w)=g(\Gamma)+\sum_{v\in V(G)}w(v),$$ which equals clearly to the genus $g({\Gamma}_{\varepsilon}^w)$ of the pure metric graph $\Gamma_{\varepsilon}^w.$\end{definition}\par

It is shown in \cite{SK-2015} that to a vertex-weighted graph $(G, w)$, one can naturally associate a vertex-weighted metric graph $(\Gamma, w)$ as follows: define $\Gamma$ to be the metric graph associated to $G$, where each edge of $G$ is assigned length $1,$ and extend the weight function $w: V(G)\to \mathbb{Z}_{\geq 0}$ to $w: {\Gamma} \to \mathbb{Z}_{\geq 0}$ by assigning $w(v)=0$ for any $x\in {{\Gamma} \setminus V(G)}.$ Then $\Gamma^w:=\Gamma_{1}^w$ for $\varepsilon=1$ in particular is the pure metric graph associated to $G^w$ (i.e., each edge of $G^w$ is assigned length $1$), and we have $g(G^w)=g(G,w)=g(\Gamma^w)=g(\Gamma,w).$\par

A divisor $D$ on a vertex-weighted metric graph $(\Gamma,w)$ is an element of the free Abelian group $Div(\Gamma,w)$ generated by points of $(\Gamma,w)=(G, w, \ell )$ defined by $D=\sum_{x\in (\Gamma,w)} D(x)(x),$ where $D(x)\in \mathbb{Z}.$ The degree of a divisor $D$ is defined by $\deg(D)=\sum_{x\in (\Gamma, w)}D(x)$. If $D(x)\geq 0$ for all $x\in (G, \ell),$ then the divisor is called effective. \par

\begin{definition}The canonical divisor of a vertex-weighted metric graph $(\Gamma, w)$ is $$K_{(\Gamma,w)}:=K_{\Gamma}+\sum_{v \in V(G)}2w(v)=\sum_{v \in V(G)}(val(v)-2+2w(v))(v).$$ \end{definition}

It follows from the Riemann-Roch theorem for vertex-weighted metric graphs \cite[Theorem 5.4]{AC-2013} that $$\deg(K_{(\Gamma,w)})=2g(\Gamma,w)-2.$$ \par

\subsection{Theory of divisors for metrized complexes of curves}\par

Metrized complexes of curves \cite{ABBR1-2015, ABBR2-2015, AB-2015, BJ-2016} can be considered as objects which interpolate between classical and tropical algebraic geometry. The theory of divisors on metrized complexes of curves generalizes both the classical theory for algebraic curves and the corresponding theory for metric graphs. The former corresponds to the case where $G$ consists of a single vertex $v$ and no edge and $\mathcal{C}_v$ is an arbitrary smooth curve. The latter corresponds to the case where the curves $\mathcal{C}_v$ have genus zero for all $v\in V(G).$ In the following statements, we consider metrized complexes of algebraic curves on an algebraically closed field $\kappa$ whose underlying metric graph is vertex-weighted. \par

\begin{definition}Let $\kappa$ be an algebraically closed field. A metrized complex $\mathfrak{C}$ of $\kappa$-curves consists of the following data:\par
\begin{itemize}
  \item A connected finite graph $G$ with vertex set $V(G)$ and edge set $E(G).$
  \item A vertex-weighted metric graph $(\Gamma, w')$ having a model $(G, \ell),$ where the length function is $\ell: E(G)\rightarrow\mathbb{R}_{>0}.$
  \item For each vertex $v\in V(G),$ a complete, nonsingular, irreducible algebraic curves $\mathcal{C}_v$ over $\kappa.$
  \item For each vertex $v\in V(G),$ a bijection $red_v:$ $e\to x_v^e$ between the edges of $G$ incident to $v$ (with loop edges counted twice) and a subset $\mathcal{A}_v=\{x_v^e\}_{v\in e}$ of $\mathcal{C}_{v}(\kappa).$
\end{itemize}
\end{definition}

For example, a metrized complex $\mathfrak{C}$ over complex field $\mathbb{C}$ can be visualized as a collection of compact Riemann surfaces connected together via real line segments. \par

\begin{definition}The geometric realization $|\mathfrak{C}|$ of a metrized complex $\mathfrak{C}$ over $\kappa$ is defined to be the union of the edges of $G$ and the collection of the curves ${\mathcal{C}}_v,$ with each endpoint $v$ of an edge $e$ identified with the corresponding marked point $x_v^e.$  (see Figure \ref{fig.1}). \end{definition}

\begin{figure}
  \centering
  \includegraphics[width=8cm]{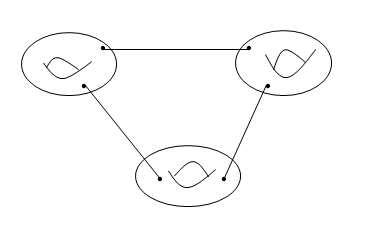}\\
  \caption{The geometric realization of a metrized complex of genus four }\label{fig.1}
\end{figure}

When we think of $|\mathfrak{C}|$ as a set, we identify it with the disjoint union of $\Gamma\setminus V(G)$ and $\cup_{v\in V(G)}\mathcal{C}_{v}(\kappa).$ Hence, if we write $x\in |\mathfrak{C}|,$ it means that $x$ is either a non-vertex point of $\Gamma$ (a graphical point of $\mathfrak{C}$) or a point of ${\mathcal{C}}_v(\kappa)$ for some $v\in V(G)$ (a geometric point of $\mathfrak{C}$). We introduce the genus of a metrized complex of algebraic curves whose underlying metric graph is vertex-weighted as follows. \par

\begin{definition}\label{D2.12}The genus of a metrized complex $\mathfrak{C}$ of $\kappa$-curves is defined as $$g(\mathfrak{C}, w)=g(\Gamma, w)+\sum_{v\in V(G)}g_v,$$ where $g_v$ is the genus of the curve ${\mathcal{C}}_v$ and $g(\Gamma, w)=g(\Gamma)+\sum_{v\in V(G)}w(v)$ is the genus of the vertex-weighted metric graph $\Gamma.$ If without the weight $w(v)$ for all vertex $v\in V(G),$ then we use the short notation $g(\mathfrak{C})$ as in \cite{AB-2015} instead of $g(\mathfrak{C}, w),$ that is $$g(\mathfrak{C})=g(\Gamma)+\sum_{v\in V(G)}g_v.$$
\end{definition}

\emph{A divisor $D$ on a metrized complex} $\mathfrak{C}$ of $\kappa$-curves is an element of the free Abelian group on $|\mathfrak{C}|,$ that is, $$D=\sum_{x\in |\mathfrak{C}|} \alpha_{x}(x),$$ where $\alpha_{x}\in \mathbb{Z},$ all but finitely many of the $\alpha_{x}$ are zero and the sum is over all points of $\Gamma\setminus V(G)$ as well as $\mathcal{C}_{v}(\kappa)$ for $v\in V(G).$ The degree of a divisor $D$ is defined by $\deg(D)=\sum_{x\in |\mathfrak{C}|}\alpha_{x}.$\par

To a divisor $D$ on  $\mathfrak{C}$, it naturally associates a divisor $D_v$ on ${\mathcal{C}}_v$ for each $v\in V(G),$ called the \emph{${\mathcal{C}}_v$-part of the divisor} $D$ which is simply the restriction of $D$ to ${\mathcal{C}}_v,$ i.e. $$D_v=\sum_{x\in {\mathcal{C}}_v(\kappa)}D(x)(x),$$ where $D(x)$ is the coefficient of $x$ in $D.$ As well as, we can associate a divisor $D_{\Gamma}$ on $\Gamma,$ called the \emph{$\Gamma$-part of the divisor} $D,$ defined by $$D_{\Gamma}=\sum_{x\in \Gamma\setminus V(G)}D(x)(x)+\sum_{v\in V(G)}\deg(D_v)(v).$$   It is easy to deduce that $$\deg(D)=\deg (D_{\Gamma}).$$ Therefore, one could equivalently define a divisor on $\mathfrak{C}$ to be an element of the form $D=D_{\Gamma}\oplus \sum_{v}D_{v}$ of $Div(\Gamma)\oplus(\oplus_{v}Div(\mathcal{C}_{v}))$ such that $\deg(D)=\deg (D_{\Gamma})$ for all $v\in V(G).$\par

We introduce the canonical divisor on metrized complexes whose underlying metric graph is vertex-weighted.\par

\begin{definition}\label{D2.13}
The canonical divisor on $\mathfrak{C}$ is defined to be the linear equivalence class of the divisor  \begin{eqnarray*}{\mathcal{K}}_{(\mathfrak{C}, w)}&:=&\sum_{v \in V(G)}(K_v+\sum_{x\in {\mathcal{A}_v}} x)+\sum_{v \in V(G)}2w(v)(v)\\&=&\sum_{v \in V(G)}(K_v+A_v)+\sum_{v \in V(G)}2w(v)(v),\end{eqnarray*} where $K_v$ denotes a divisor of degree $2g_{v}-2$ in the canonical class of ${\mathcal{C}}_v$ and $A_v$ is the divisor in $\mathcal{C}_{v}$ consisting of the sum of the $val(v)$ points in ${\mathcal{A}}_v.$ If without the weight $w(v)$ for all vertex $v\in V(G),$ then we use the short notation ${\mathcal{K}}_{\mathfrak{C}}$ as in \cite{AB-2015} instead of ${\mathcal{K}}_{(\mathfrak{C}, w)},$ that is $${\mathcal{K}}_{\mathfrak{C}}=\sum_{v \in V(G)}(K_v+A_v).$$
The $\mathcal{C}_v$-part and $\Gamma$-part of the canonical divisor  ${\mathcal{K}}_{\mathfrak{C}}$ are respectively defined by $$K_{\mathcal{C}_{v}}:=K_v+A_v \quad\mbox{and}\quad K^{\#}:=\sum_{v\in V(G)}(val(v)+2g_v-2)(v)=K_{\Gamma}+\sum_{v\in V(G)}2g_v.$$\end{definition}

It follows from the Riemann-Roch theorem for metrized complexes of algebraic curves \cite[Theorem 1.4]{AB-2015} that \begin{equation}\label{Eq0}\deg({\mathcal{K}}_{\mathfrak{C}})=2g(\mathfrak{C})-2.\end{equation}
Then by the definition of the genus of a metrized complex of curves $\mathfrak{C}$ of $\kappa$-curves, we have $$\deg({\mathcal{K}}_{\mathfrak{C}})=2g(\Gamma)-2+\sum_{v\in V(G)}2g_v=\sum_{v\in V(G)}(\deg_G(v)+2g_v-2)=\deg(K^{\#}).$$

Hence, it follows from Definition \ref{D2.12}, Definition \ref{D2.13} and \eqref{Eq0} that
\begin{eqnarray}\nonumber\deg({\mathcal{K}}_{(\mathfrak{C}, w)})&=&\deg({\mathcal{K}}_{\mathfrak{C}})+\sum_{v \in V(G)}2w(v)\\
&=&2g(\mathfrak{C})-2+\sum_{v \in V(G)}2w(v)\\\nonumber
&=&2\left(g(\Gamma)+\sum_{v\in V(G)}g_v\right)-2+\sum_{v \in V(G)}2w(v)\\\nonumber
&=&2g(\Gamma, w)+2\sum_{v\in V(G)}g_v-2\\\nonumber
&=&2g(\mathfrak{C}, w)-2.\end{eqnarray}

\section{Second main theorem on finite graphs}
\subsection{Harmonic morphism between finite graphs}\par

As shown in the previous section, we can turn a finite graph $G$ with loops into one $\hat{G}$ without loops. We will only discuss harmonic morphism between finite graphs without loops throughout this section.\par
Let $G$ and $G'$ be two finite graphs. A function $\phi:V(G)\cup E(G) \to V(G')\cup E(G')$ is said to be a morphism from $G$ to $G'$ if $\phi(V(G))\subseteq V(G')$, and for every edge $e\in E(G)$ with endpoints $v_1$ and $v_2$, either $\phi(e)\in E(G')$ and $\phi(v_1)$, $\phi(v_2)$ are the endpoints of $\phi(e)$, or $\phi(e)\in V(G')$ and $\phi(v_1)=\phi(e)=\phi(v_2)$. If $\phi(E(G)) \subseteq E(G'),$ then morphism $\phi$ is called a homomorphism. A bijective homomorphism is called an isomorphism, and an automorphism $\phi:G\to G$ is a bijective homomorphism.\par

Let $\phi:G\to G'$ be a morphism. For each vertex $v\in V(G)$, the \emph{vertical multiplicity} $V_{\phi}(v)$ is defined to be the number of \emph{vertical edges} incident to $v,$ that is $$V_{\phi}(v)=|\{e\in E(G): v\in e, \phi(e)=\phi(v)\}|.$$\par

Harmonic morphisms between simple graphs were built by Urakawa \cite{Ura-2000}. The following definitions and lemmas are proposed by Baker and Norine \cite{BN-2009}, including the definition of harmonic morphism as the direct graph analogue of a holomorphic map between Riemann surfaces:\par

\begin{definition} A morphism $\phi:G\to G'$ is said to be harmonic (horizontally conformal) if, for all $v\in V(G),$ $v'\in V(G')$ such that $v'=\phi(v),$ the quantity $|\{e\in E(G)|v\in e, \phi(e)=e'\}|$ is the same for all edges $e'\in E(G')$ such that $v'\in e'$.
\end{definition}

Let $\phi$ be harmonic. The \emph{horizontal multiplicity} of $\phi$ at each $v\in V(G)$ is given by
\begin{eqnarray*}
M_{\phi}(v)=\left\{
  \begin{array}{ll}
    0, & |V(G')|=1; \\
    |\{e\in E(G): v\in e, \phi(e)=e'\}|\, \mbox{for any}\\\mbox{edge}\, e'\in E(G')\, \mbox{such that}\, \phi(v)\in e', & |V(G')|>1.
  \end{array}
\right.
\end{eqnarray*}

We have a basic formula relating the horizontal and vertical multiplicities as follows:
\begin{equation*}val(v)=val(\phi(v))M_{\phi}(v)+V_{\phi}(v)\end{equation*} for any vertex $v\in V(G).$

If $M_{\phi}(v)\geq 1$ for every $v\in V(G),$ then we say the harmonic morphism $\phi: G\to G'$ is \emph{nondegenerate}.

The \emph{degree of the harmonic morphism} $\phi: G\to G'$ is defined as
\begin{eqnarray*}
\deg(\phi)=\left\{
  \begin{array}{ll}
    0, & |V(G')|=1; \\
    |\{e\in E(G): \phi(e)=e'\}| \,\mbox{for any edge}\, e'\in E(G'), & |V(G')|>1.
  \end{array}
\right.
\end{eqnarray*}
\par

\begin{figure}
  \centering
  \includegraphics[width=10cm]{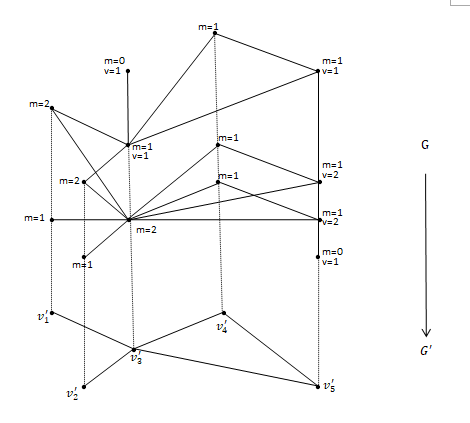}\\
  \caption{A harmonic morphism $\phi:G\to G'$ of degree three}\label{fig.2}
\end{figure}

It is proved \cite[Lemma 2.3]{BN-2009} that the degree $\deg(\phi)$ is independent of the choice of the edge $e'\in E(G').$  The following lemma says that the degree of a harmonic morphism $\phi: G\to G'$ is just the number of pre-images under $\phi$ of any vertex of $G',$ counting multiplicities.\par

\begin{lemma}\cite[Lemma 2.3]{BN-2009} \label{L3.2} For any vertex $v'\in V(G'),$ we have $$\deg(\phi)=\sum_{v\in V(G)\atop \phi(v)=v'}M_{\phi}(v).$$\end{lemma}

A harmonic morphism of finite graphs must be either constant or surjective, as with morphisms of Riemann surfaces in algebraic geometry.\par

\begin{lemma}\cite[Lemma 2.4]{BN-2009} Let $\phi: G\to G'$ be a harmonic morphism with $|V(G')|>1.$ Then $\deg(\phi)=0$ if and only if $\phi$ is constant, and $\deg(\phi)>0$ if and only if $\phi$ is surjective.\end{lemma}

\begin{definition}
Let $\phi: G\to G'$ be a harmonic morphism, the pullback homomorphism $\phi^*: Div(G')\to Div(G)$ is defined by $$\phi^*(D')=\sum_{v'\in V(G')}\sum_{v\in V(G)\atop \phi(v)=v'}M_{\phi}(v)D'(v')(v)=\sum_{v\in V(G)}M_{\phi}(v)D'(\phi(v))(v).$$ The push-forward homomorphism $\phi_{*}: Div(G)\to Div(G')$ is similarly defined by
$$\phi_{*}(D)=\sum_{v\in V(G)}D(v)(\phi(v)).$$
\end{definition}

\begin{lemma}\cite[Lemma 2.8]{BN-2009} If $\phi: G\to G'$ is a harmonic morphism and $D'\in Div(G')$, then $$\deg(\phi^*(D'))=\deg(\phi)\deg(D').$$ \end{lemma}

\begin{lemma}\cite[Lemma 4.1]{BN-2009} Let $\phi: G\to G'$ be a harmonic morphism and $D'\in Div(G')$, then
$$\phi_{*}(\phi^{*})(D')=\deg(\phi)D'.$$\end{lemma}

\subsection{Riemann-Hurwitz theorem on finite graphs}
In 2009, Baker and Norine firstly obtained the Riemann -Hurwitz theorem on finite graphs as follows.\par

\begin{theorem}\cite[Theorem 2.9]{BN-2009}\label{T3.7} Let $G$, $G'$ be finite graphs, and let $\phi:G\to G'$ be a harmonic morphism. Then:\par

(i). The canonical divisors on $G$ and $G'$ are related by the formula $$K_G=\phi^*K_{G'}+R_G.$$ where $$R_G=2\sum_{v\in V(G)}(M_{\phi}(v)-1)(v)+\sum_{v\in V(G)}V_{\phi}(v)(v).$$\par

(ii). If $G$, $G'$ have genus $g$ and $g'$, respectively, then $$2g-2=\deg(\phi)(2g'-2)+\sum_{v\in V(G)}(2(M_{\phi}(v)-1)+V_{\phi}(v)).$$\par

(iii). If $\phi$ is nonconstant, then $2g-2\geq \deg(\phi)(2g'-2)$ and $g\geq g'$.\par
\end{theorem}

\subsection{Second main theorem on finite graphs}\par
We now establish the second main theorem for harmonic morphisms on finite graphs in terms of the Riemann-Hurwitz theorem for finite graphs, from the viewpoint of Nevanlinna theory for algebraic curves in Riemann surfaces. \par

\begin{theorem} Let $G$, $G'$ be finite graphs with genus $g$ and $g'$, respectively, and let $\phi:G\to G'$ be a harmonic morphism. Suppose that $a_1, \ldots, a_q\in V(G')$ be distinct vertices, and let $E=\phi^{-1}(\{a_1, \ldots, a_q\}).$ Then $$(q+g'-1)\deg(\phi)\leq g-1+|E\cap V(G)|-\frac{1}{2}\sum_{v\in V(G)}V_{\phi}(v).$$ where $|E\cap V(G)|$ is the cardinality of the set $E\cap V(G).$
\end{theorem}

\begin{proof}
Set $$r_{\phi}(E):=\sum_{v\in E\cap V(G)}(M_{\phi}(v)-1),$$ and $$r_{\phi}(G):=\sum_{v\in V(G)}(M_{\phi}(v)-1).$$ It is obvious that $$r_{\phi}(E)\leq r_{\phi}(G).$$ From the definitions of the degree of a harmonic morphism $\phi: G\to G'$, the horizontal multiplicity of $\phi$ at $v$ and Lemma \ref{L3.2}, we  get that $$\sum_{v\in E\cap V(G)\atop a_j=\phi(v)}M_{\phi}(v)=\sum_{v\in V(G)\atop a_j=\phi(v)}M_{\phi}(v)=\deg(\phi)$$ holds for each $ a_j\in \{a_1, \ldots, a_q\}.$
Then we have \begin{eqnarray*}r_{\phi}(E)&=&\sum_{v\in E\cap V(G)}(M_{\phi}(v)-1)\\\nonumber&=&(\sum\limits_{j=1}^q \sum\limits_{v\in E\cap V(G)\atop a_j=\phi(v)}M_{\phi}(v))-|E\cap V(G)|\\\nonumber&=&q\deg(\phi)-|E\cap V(G)|.\end{eqnarray*}  On the other hand, by the Riemann-Hurwitz theorem for finite graphs (Theorem \ref{T3.7}), we have
$$r_{\phi}(G)=g-1-(g'-1)\deg(\phi)-\frac{1}{2}\sum_{v\in V(G)}V_{\phi}(v).$$ Hence, we get the following inequality $$(q+g'-1)\deg(\phi)\leq g-1+|E\cap V(G)|-\frac{1}{2}\sum_{v\in V(G)}V_{\phi}(v).$$
\end{proof}

We give some examples to explain the second main theorem for finite graphs. One can refer to the article \cite{BN-2009} for more examples.\par

\begin{example} In the above Figure \ref{fig.2} \cite[Example 3.1]{BN-2009}, with horizontal and vertical multiplicities $M_{\phi}(v)$ and $V_{\phi}(v)$, are shortly written as $m$ and $v$, respectively, labeled next to the corresponding vertices. It is easy to calculate that:\\
$$g=|E(G)|-|V(G)|+1=6,\quad g'=|E(G')|-|V(G')|+1=1,\quad \deg(\phi)=3.$$\par
\begin{itemize}
  \item Take $q=5$ and let $\{a_1, \ldots, a_q\}=\{v_{1}', \ldots, v_{5}'\}=V(G^{'}).$ Then one can get that $|E\cap V(G)|=14$ and $\frac{1}{2}\sum\limits_{v\in V(G)}V_{\phi}(v)=4.$ By the second main theorem we get that $(5+1-1)\times 3\leq 6-1+14-4,$ which in fact is an equality. This means that the inequality of the second main theorem is sharp.

  \item Take $q=4.$ \begin{itemize}
                      \item Let $\{a_1, \ldots, a_q\}=\{v_{1}',v_{3}',v_{4}', v_{5}'\}\subset V(G')$ (or $=\{v_{2}',v_{3}',v_{4}', v_{5}'\}$). Then $|E\cap V(G)|=12$ and $\frac{1}{2}\sum\limits_{v\in V(G)}V_{\phi}(v)=4$, so we have the inequality $12\leq 5+12-4=13$ by the second main theorem.

                      \item Let $\{a_1, \ldots, a_q\}= \{v_{1}',v_{2}',v_{3}', v_{4}'\}\subset V(G')$. Then $|E\cap V(G)|=10$, and $\frac{1}{2}\sum\limits_{v\in V(G)}V_{\phi}(v)=1$, so we get the inequality $12 \leq 5+10-1=14$ by the second main theorem.

                      \item Let $\{a_1, \ldots, a_q\}= \{v_{1}',v_{2}',v_{3}', v_{5}'\}\subset V(G')$. Then $|E\cap V(G)|=11$, and $\frac{1}{2}\sum\limits_{v\in V(G)}V_{\phi}(v)=4$, so we have the equality $12=5+11-4=12$ by the second main theorem.

                      \item Let $\{a_1, \ldots, a_q\}= \{v_{1}',v_{2}',v_{4}', v_{5}'\}\subset V(G')$. Then $|E\cap V(G)|=11$, and $\frac{1}{2}\sum\limits_{v\in V(G)}V_{\phi}(v)=3$, so we get the inequality $12 \leq 5+11-3=13$ by the second main theorem.
                    \end{itemize}
  \item Take $q=3$. \begin{itemize}
                      \item Let $\{a_1, \ldots, a_q\}=\{v_{1}',v_{3}',v_{4}'\}\subset V(G')$ (or $=\{v_{2}',v_{3}',v_{4}'\}$). Then $|E\cap V(G)|=8$, and $\frac{1}{2}\sum\limits_{v\in V(G)}V_{\phi}(v)=1$, so we have the inequality $9 \leq 5+8-1=12$ by the second main theorem.

                      \item Let $\{a_1, \ldots, a_q\}=\{v_{1}',v_{3}',v_{5}'\}\subset V(G')$ (or $=\{v_{2}',v_{3}',v_{5}'\}$). Then$|E\cap V(G)|=9$, and $\frac{1}{2}\sum\limits_{v\in V(G)}V_{\phi}(v)=4$, so we have the inequality $9 \leq 5+9-4=10$ by the second main theorem.

                      \item Let $\{a_1, \ldots, a_q\}=\{v_{1}',v_{4}',v_{5}'\}\subset V(G')$ (or $=\{v_{2}',v_{4}',v_{5}'\}$). Then $|E\cap V(G)|=9$, and $\frac{1}{2}\sum\limits_{v\in V(G)}V_{\phi}(v)=3$, so we have the inequality $9 \leq 5+9-3=11$ by the second main theorem.

                      \item Let $\{a_1, \ldots, a_q\}=\{v_{1}',v_{2}',v_{3}'\}\subset V(G')$. Then $|E\cap V(G)|=7$, and $\frac{1}{2}\sum\limits_{v\in V(G)}V_{\phi}(v)=1$, so we have the inequality $9 \leq 5+7-1=11$ by the second main theorem.

                      \item Let $\{a_1, \ldots, a_q\}=\{v_{1}',v_{2}',v_{4}'\}\subset V(G')$. Then $|E\cap V(G)|=7$, and $\frac{1}{2}\sum\limits_{v\in V(G)}V_{\phi}(v)=0$, so we have the inequality $9 \leq 5+7=12$ by the second main theorem.

                      \item Let $\{a_1, \ldots, a_q\}=\{v_{1}',v_{2}',v_{5}'\}\subset V(G')$. Then$|E\cap V(G)|=8$, and $\frac{1}{2}\sum\limits_{v\in V(G)}V_{\phi}(v)=3$, so we have the inequality $9 \leq 5+8-3=10$ by the second main theorem.

                      \item Let $\{a_1, \ldots, a_q\}=\{v_{3}',v_{4}',v_{5}'\}\subset V(G')$. Then $|E\cap V(G)|=10$, and $\frac{1}{2}\sum\limits_{v\in V(G)}V_{\phi}(v)=4$, so we know have inequality $9 \leq 5+10-4=11$ by the second main theorem.

                    \end{itemize}
  \item  Take $q=2$, \begin{itemize}
                       \item  Let $\{a_1, a_q\}=\{v_{1}',v_{3}'\}\subset V(G')$ (or $=\{v_{2}',v_{3}'\}, \{v_{3}',v_{4}'\}$). Then $|E\cap V(G)|=5$, and $\frac{1}{2}\sum\limits_{v\in V(G)}V_{\phi}(v)=1$, so we have the inequality $6 \leq 5+5-1=9$ by the second main theorem.
                       \item Let $\{a_1, a_q\}=\{v_{1}',v_{4}'\}\subset V(G')$ (or $=\{v_{2}',v_{4}'\}$). Then $|E\cap V(G)|=5$, and $\frac{1}{2}\sum\limits_{v\in V(G)}V_{\phi}(v)=0$, so we have the inequality $6 \leq 5+5=10$ by the second main theorem.
                       \item Let $\{a_1, a_q\}=\{v_{1}',v_{5}'\}\subset V(G')$ (or $=\{v_{2}',v_{5}'\}$). Then $|E\cap V(G)|=6$, and $\frac{1}{2}\sum\limits_{v\in V(G)}V_{\phi}(v)=3$, so we have the inequality $6 \leq 5+6-3=8$ by the second main theorem.
                       \item Let $\{a_1, a_q\}=\{v_{1}',v_{2}'\}\subset V(G')$. Then $|E\cap V(G)|=4$, $\frac{1}{2}\sum\limits_{v\in V(G)}V_{\phi}(v)=0$, so we have the inequality $6 \leq 5+4=9$ by the second main theorem.
                       \item Let $\{a_1, a_q\}=\{v_{3}',v_{5}'\}\subset V(G')$. Then $|E\cap V(G)|=7$, $\frac{1}{2}\sum\limits_{v\in V(G)}V_{\phi}(v)=4$, so we have the inequality $6 \leq 5+7-4=8$ by the second main theorem.
                       \item Let $\{a_1, a_q\}=\{v_{4}',v_{5}'\}\subset V(G')$. Then $|E\cap V(G)|=7$, $\frac{1}{2}\sum\limits_{v\in V(G)}V_{\phi}(v)=3$, so we have the inequality $6 \leq 5+7-3=9$ by the second main theorem.
                     \end{itemize}
  \item Take $q=1$, \begin{itemize}
                      \item Let $a_q\in V(G')$ take $v_{1}'$, or, $v_{2}'$. Then $|E\cap V(G)|=2$, $\frac{1}{2}\sum\limits_{v\in V(G)}V_{\phi}(v)=0$, so we have the inequality $3 \leq 5+2=7$ by the second main theorem.
                      \item Let $a_q\in V(G')$ take $v_{3}'$. Then $|E\cap V(G)|=3$, and $\frac{1}{2}\sum\limits_{v\in V(G)}V_{\phi}(v)=1$, so we have the inequality $3 \leq 5+3-1=7$ by the second main theorem.
                      \item Let $a_q\in V(G')$ take $v_{4}'$. Then $|E\cap V(G)|=3$, and $\frac{1}{2}\sum\limits_{v\in V(G)}V_{\phi}(v)=0$, so we have the inequality $3 \leq 5+3=8$ by the second main theorem.
                      \item Let $a_q\in V(G')$ take $v_{5}'$. Then $|E\cap V(G)|=4$, and $\frac{1}{2}\sum\limits_{v\in V(G)}V_{\phi}(v)=3$, so we have the inequality $3 \leq 5+4-3=6$ by the second main theorem.
                    \end{itemize}
\end{itemize}
\end{example}

\begin{example} [collapsing] Let $p\in V(G)$ be a cut vertex, so that $G$ can be partitioned into two subsets $G_1$ and $G_2$, which intersect only at $p$, the collapsing of $G$ relative to $G_1$ is the graph $G'$ obtained by contracting all vertices and edges in $G_1$ to $\{p\}$. Let $ \phi:G\to G'$ be the morphism that sends $G_1$ to $p$ and is the identity on $G_2$, if $|V(G_2)|>1$, then $\phi$ is a harmonic morphism.  See Figure \ref{fig.3}, the horizontal multiplicities $M_{\phi}(v)$ and vertical multiplicities $V_{\phi}(v)$, are written as $m$ and $v$, respectively, labeled next to the corresponding vertices. It is known by calculation that:\\
$g=|E(G)|-|V(G)|+1=0,$ and $g'=|E(G')|-|V(G')|+1=0,$ $\deg(\phi)=1.$\par
\begin{itemize}
  \item  Take $q=5$ and let $\{a_1, \ldots, a_q\}=\{v_{1}', \ldots, v_{5}'\}=V(G').$ Then we can get that $|E\cap V(G)|=6$ and $\frac{1}{2}\sum\limits_{v\in V(G)}V_{\phi}(v)=1.$ By the second main theorem we get that $(5+0-1)\times 1\leq 0-1+6-1.$ This means that the inequality of the second main theorem is sharp.
  \item Take $q=4$, \begin{itemize}
                      \item Let $\{a_1, \ldots, a_q\}=\{v_{1}',v_{3}',v_{4}', v_{5}'\}\subset V(G')$ (or can take $\{v_{2}',v_{3}',v_{4}', v_{5}'\}$, $\{v_{1}',v_{2}',v_{3}', v_{4}'\},$ $\{v_{1}',v_{2}',v_{3}', v_{5}'\}$). Then $|E\cap V(G)|=5$, $\frac{1}{2}\sum\limits_{v\in V(G)}V_{\phi}(v)=1$, so we have the equality  $3=-1+5-1=3$ by the second main theorem.
                      \item Let $\{a_1, \ldots, a_q\}=\{v_{1}',v_{2}',v_{4}', v_{5}'\}\subset V(G')$. Then $|E\cap V(G)|=4$ and $\frac{1}{2}\sum\limits_{v\in V(G)}V_{\phi}(v)=0$, so we have the equality  $3=-1+4=3$ by the second main theorem.
                    \end{itemize}
  \item Take $q=3$,\begin{itemize}
                     \item Let $\{a_1, \ldots, a_q\}=\{v_{1}',v_{3}',v_{4}'\}\subset V(G')$ (or $=\{v_{2}',v_{3}',v_{4}'\}$, $\{v_{1}',v_{2}',v_{3}' \},$ $\{v_{1}',v_{3}', v_{5}'\}$, $\{v_{2}',v_{3}',v_{5}'\}$, $\{v_{3}',v_{4}',v_{5}' \},$). Then $|E\cap V(G)|=4$ and $\frac{1}{2}\sum\limits_{v\in V(G)}V_{\phi}(v)=1$, so we have the equality  $2=-1+4-1=2$ by the second main theorem.
                     \item Let $\{a_1, \ldots, a_q\}=\{v_{1}',v_{4}',v_{5}'\}\subset V(G')$ (or $=\{v_{2}',v_{4}',v_{5}'\}$, $\{v_{1}',v_{2}',v_{4}' \},$ $\{v_{1}',v_{2}', v_{5}'\}$). Then $|E\cap V(G)|=3$ and $\frac{1}{2}\sum\limits_{v\in V(G)}V_{\phi}(v)=0$, so we have the equality  $2=-1+3=2$ by the second main theorem.
                   \end{itemize}
  \item Take $q=2$,\begin{itemize}
                     \item Let $\{a_1, a_q\}=\{v_{1}',v_{3}'\}\subset V(G')$ (or $=\{v_{2}',v_{3}'\}$, $\{v_{3}',v_{4}'\}$,  $\{v_{3}',v_{5}'\}$). Then
                      $|E\cap V(G)|=3$, and $\frac{1}{2}\sum\limits_{v\in V(G)}V_{\phi}(v)=1$, so we have the equality  $1=-1+3-1=1$ by the second main theorem.
                     \item Let $\{a_1, a_q\}=\{v_{1}',v_{2}'\}\subset V(G')$ (or $=\{v_{1}',v_{4}'\}$, $\{v_{1}',v_{5}'\}$,  $\{v_{2}',v_{4}'\}$,  $\{v_{2}',v_{5}'\}$,  $\{v_{4}',v_{5}'\}$). Then $|E\cap V(G)|=2$, and $\frac{1}{2}\sum\limits_{v\in V(G)}V_{\phi}(v)=0$, so we have the equality  $1=-1+2=1$ by the second main theorem.
                   \end{itemize}
  \item Take $q=1$, \begin{itemize}
                      \item Let $a_q\in V(G')$ take $v_{1}'$, or, $v_{2}'$, or, $v_{1}'$, or, $v_{2}'$, then $|E\cap V(G)|=1$, and $\frac{1}{2}\sum\limits_{v\in V(G)}V_{\phi}(v)=0$, so we have the equality $0=-1+1=0$  by the second main theorem.
                      \item Let $a_q\in V(G')$ take $v_{3}'$, then $|E\cap V(G)|=2$, and $\frac{1}{2}\sum\limits_{v\in V(G)}V_{\phi}(v)=1$, so we have the equality $0=-1+2-1=0$ by the second main theorem.
                    \end{itemize}
\end{itemize}
\end{example}
\begin{figure}
  \centering
  \includegraphics[width=8cm]{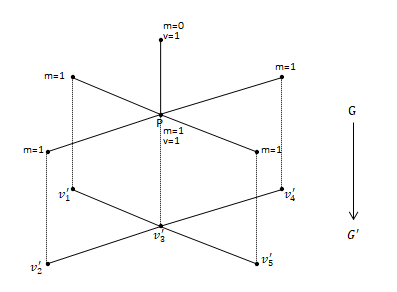}\\
  \caption{A harmonic morphism $\phi:G\to G'$ of degree one }\label{fig.3}
\end{figure}

\section{Second main theorem for vertex-weighted graphs}

\subsection{Pseudo-harmonic indexed morphism between vertex-weighted graphs}\par
Recall that in the subsection 2.2, for a vertex-weighted graph $(G, w),$ we can associate to it a \emph{weightless graph} $G^w,$ which is obtained by attaching at every vertex $v\in V(G),$ $w(v)$ loops (or ``$1$-cycles"). Conversely, when we deal with a vertex-weighted  graph with loops, we can add the weighted number corresponding to the number of loops at that point. So we will only suppose that all vertex-weighted graphs are loopless in this section.\par

Let $\phi: G\to G'$ be a morphism, and denote by ${\phi}_{V}: V(G)\to V(G')$ the map induced by $\phi$ on the vertices. A morphism between vertex-weighted  graphs $(G,w)$ and $(G',w')$ is defined as a morphism of the underlying graphs. The following definitions and lemmas are proposed in \cite{Cap-2014}, extending the ones in Section 3 (\cite{BN-2009}).\par

\begin{definition} Let $(G,w)$ and $(G',w')$ be loopless vertex-weighted graphs.\par
(i) A indexed morphism is a morphism $\phi: (G,w)\to (G',w')$ enriched by the assignment, for every $e\in E(G)$, assign a non-negative integer, the index of $\phi$ at $e$, written $r_{\phi}(e)$, such that  $r_{\phi}(e)=0$ if and only if $\phi(e)$ is a point. If for every $e\in E(G),$ $r_{\phi}(e)\leq 1$, then the indexed morphism is simple.\par

(ii) A indexed morphism is pseudo-harmonic if for  every $v\in V(G)$ there exists a number, $M_{\phi}(v)$, such that for every $e'\in E(G')$ we have $$M_{\phi}(v)=\sum_{e\in E(G): \phi(e)=e'}r_{\phi}(e).$$\par

(iii)  A pseudo-harmonic indexed morphism is non-degenerate if $M_{\phi}(v)\geq 1$, for every $v\in V(G).$\par
(iv) A pseudo-harmonic indexed morphism is harmonic if for every $v\in V(G)$ we have $$ \sum_{e\in E(G)}(r_{\phi}(e)-1)\leq 2(M_{\phi}(v)-1+w(v)-M_{\phi}(v)w'(v')).$$  where $v'=\phi (v)$.
\end{definition}

\begin{remark} For simple morphisms of weightless graphs, the above definition of harmonic morphism coincides with the one given in Section 3 (\cite{BN-2009}) for morphisms which contract no leaves.\end{remark}

\begin{definition}
 If $\phi: (G,w)\to (G',w')$ be a pseudo-harmonic indexed morphism. Then for every $e'\in E(G')$, the degree of $\phi$ as follows $$\deg\phi=\sum_{e\in E(G):\phi(e)=e'}r_{\phi}(e).$$
\end{definition}

\begin{lemma} \cite[Lemma-Definition 2.4]{Cap-2014} \label{L4.4} Let $\phi: (G,w)\to (G',w')$ be a pseudo-harmonic indexed morphism. For any vertex $v'\in V(G')$, we have $$\deg(\phi)=\sum_{v\in V(G)\atop \phi(v)=v'}M_{\phi}(v).$$  \end{lemma}

\begin{definition}\cite{Cap-2014}
Let $\phi: (G,w)\to (G',w')$ be a pseudo-harmonic indexed morphism, the pull-back homomorphism $\phi^*: Div(G',w')\to Div(G,w)$ as follows: for every $v'\in V(G')$, $$\phi^*(D')=\sum_{v \in {\phi}^{-1}(v')}M_{\phi}(v)D'(\phi(v))(v).$$
\end{definition}

\begin{lemma} \cite{Cap-2014}\label{L4.6} If $\phi: (G,w)\to (G',w')$ be a pseudo-harmonic indexed morphism, then $$\deg(\phi^*(D'))=\deg(\phi)\deg(D').$$ \end{lemma}

\subsection{Riemann-Hurwitz theorem on vertex-weighted graphs}\par
In 2014, Caporaso extended the Riemann-Hurwitz theorem to the case of vertex-weighted graphs.\par

\begin{theorem}\cite[Proposition 2.5]{Cap-2014} \label{T4.7} Let $(G,w)$, $(G',w')$ be loopless vertex-weighted graphs, and let $\phi:(G,w)\to (G',w')$ be a pseudo-harmonic indexed morphism of vertex-weighted graphs of genus $g$ and $g'$ respectively. Then $$K_{(G,w)}=\phi^*K_{(G',w')}+R_{\phi},$$ where and \begin{eqnarray*}
R_{\phi}=\sum_{v\in V(G)}2(M_{\phi}(v)-1+w(v)-M_{\phi}(v)w'(v'))(v)-\sum\limits_{v\in V(G)\atop e\in  {E(G)}}(r_{\phi}(e)-1)(v),
\end{eqnarray*} and $v'=\phi(v).$  Furthermore, $\phi$ is harmonic if and only if $R_{\phi}\geq 0.$
\end{theorem}

For getting the second main theorem on vertex-weighted graphs, we here need supplement the Riemann-Hurwitz theorem, corresponding to the conclusion (ii) of Theorem \ref{T3.7}.\par

\begin{theorem} \label{T4.8} If $(G,w)$, $(G',w')$ have genus $g$ and $g'$, respectively, then

\begin{eqnarray*}
2g-2 &=& \deg(\phi)(2g'-2)+\sum_{v\in V(G)}2(M_{\phi}(v)-1+w(v)-M_{\phi}(v)w'(v'))\nonumber\\
&\;&- \sum\limits_{v\in V(G)}\sum\limits_{e\in  {E_v(G)}}(r_{\phi}(e)-1)\\&=& \deg(\phi)(2g'-2)+\sum_{v\in V(G)}2(M_{\phi}(v)-1+w(v)-M_{\phi}(v)w'(v'))\nonumber\\
&\;&+\sum_{v\in V(G)}(val(v)-M_{\phi}(v)val(v')),
\end{eqnarray*} where $v'=\phi(v).$
\end{theorem}
\begin{proof} One can easily get the conclusion by Lemma \ref{L4.6}, Theorem \ref{T4.7} and the fact
$$ \sum\limits_{e\in  {E_v(G)}}(r_{\phi}(e)-1)= \sum\limits_{e\in  {E_v(G)}}r_{\phi}(e)- val(v)=  M_{\phi}(v)val(v')-val(v).$$
\end{proof}

\subsection{Second main theorem on vertex-weighted graphs}\par
Now we show the second main theorem on vertex-weighted graphs as follows.\par

\begin{theorem} Let $\phi: (G,w)\to (G',w')$ be a loopless pseudo-harmonic indexed morphism of vertex-weighted graphs of genus $g$ and $g'$, respectively. Suppose that $\{a_1, \ldots, a_q\}\subset V(G', w')$ are distinct vertices, and let $E=\phi^{-1}(\{a_1, \ldots, a_q\}).$ Then

\begin{eqnarray*}
(q+g'-1)\deg(\phi) &\leq& g-1+|E\cap V(G)|-\sum_{v\in V(G)}(w(v)-M_{\phi}(v)w'(v'))\nonumber\\
&\;&+\frac{1}{2}\sum\limits_{v\in V(G)}\sum\limits_{e\in  {E_v(G)}}(r_{\phi}(e)-1),
\end{eqnarray*}
where $|E\cap V(G)|$ is the cardinality of $E\cap V(G)$.
\end{theorem}

\begin{proof}
Set $$r_{\phi}(E):=\sum_{v\in E\cap V(G)}(M_{\phi}(v)-1),$$ and $$r_{\phi}(G, w):=\sum_{v\in V(G)}(M_{\phi}(v)-1).$$ It is obvious that $$r_{\phi}(E)\leq r_{\phi}(G,w).$$ From the definitions of the degree of a pseudo-harmonic indexed morphism $\phi:(G,w)\to (G',w')$, and Lemma \ref{L4.4}, we  get that $$\sum_{v\in E\cap V(G)\atop a_j=\phi(v)}M_{\phi}(v)=\sum_{v\in V(G)\atop a_j=\phi(v)}M_{\phi}(v)=\deg(\phi)$$ holds for each $ a_j\in \{a_1, \ldots, a_q\}.$
Then we have \begin{eqnarray*}r_{\phi}(E)&=&\sum_{v\in E\cap V(G)}(M_{\phi}(v)-1)\\\nonumber&=&(\sum\limits_{j=1}^q \sum\limits_{v\in E\cap V(G)\atop a_j=\phi(v)}M_{\phi}(v))-|E\cap V(G)|\\\nonumber&=&q\deg(\phi)-|E\cap V(G)|.\end{eqnarray*}  On the other hand, by the Riemann-Hurwitz theorem for vertex-weighted graphs (Theorem \ref{T4.8}), we have

\begin{eqnarray*}
 r_{\phi}(G,w) &=& g-1-(g'-1)\deg(\phi)-\sum_{v\in V(G)}(w(v)-M_{\phi}(v)w'(v'))\nonumber\\
&\;&+\frac{1}{2}\sum\limits_{v\in V(G)}\sum\limits_{e\in  {E_v(G)}}(r_{\phi}(e)-1).
\end{eqnarray*}
Hence, from the above inequality, this proves the second main theorem.
\end{proof}

We give some examples that satisfies the second main theorem for vertex-weighted graphs.\par

\begin{example} In the following Figure \ref{fig.4} \cite[Example 2.18]{Cap-2014},  the pseudo-harmonic indexed morphism $\phi: G\to G',$ in which one index of the edge joining $v_2$ and $v_3$ is $2$, and all other indexes of edges are $1.$ Assume that all weights are zero. Then by calculation we have:\\
$g=|E(G)|-|V(G)|+1=5,$ $g'=|E(G')|-|V(G')|+1=0,$ $\deg(\phi)=3,$ and the $\sum_{v\in V(G)}(w(v)-M_{\phi}(v)w'(v'))=0$ for each vertex.\par
\begin{itemize}
  \item Take $q=4$, \begin{itemize}
                      \item let $\{a_1, \ldots, a_q\}=\{v_{1}', \ldots, v_{4}'\}=V(G').$ Then $|E\cap V(G)|=4$ and $\frac{1}{2}\sum\limits_{v\in V(G)}\sum\limits_{e\in  {E_v(G)}}(r_{\phi}(e)-1)=1.$ By the second main theorem we get that $(4+0-1)\times 3\leq 5-1+4+1,$ which in fact is an equality. This means that the inequality of the second main theorem is sharp.
                    \end{itemize}
  \item Take $q=3$, \begin{itemize}
                     \item  Let $\{a_1, \ldots, a_q\}=\{v_{1}',v_{2}',v_{3}'\}\subset V(G')$ (or $=\{v_{2}',v_{3}',v_{4}'\}$). Then $|E\cap V(G)|=3$, and $\frac{1}{2}\sum\limits_{v\in V(G)}\sum\limits_{e\in  {E_v(G)}}(r_{\phi}(e)-1)=1$, so we have the inequality $6 \leq 4+3+1=8$ by the second main theorem.
                     \item Let $\{a_1, \ldots, a_q\}=\{v_{1}',v_{2}',v_{4}'\}\subset V(G')$ (or $=\{v_{1}',v_{3}',v_{4}'\}$). Then $|E\cap V(G)|=3$, and $\frac{1}{2}\sum\limits_{v\in V(G)}\sum\limits_{e\in  {E_v(G)}}(r_{\phi}(e)-1)=\frac{1}{2}$, so we have the inequality $6 \leq 4+3+0.5=7.5$ by the second main theorem.
                   \end{itemize}
  \item Take $q=2$, \begin{itemize}
                      \item Let $\{a_1, a_q\}=\{v_{1}',v_{2}'\}\subset V(G')$ (or $=\{v_{1}',v_{3}'\}$, $\{v_{2}',v_{4}'\}$,  $\{v_{3}',v_{4}'\}$). Then
                      $|E\cap V(G)|=2$, and $\frac{1}{2}\sum\limits_{v\in V(G)}\sum\limits_{e\in  {E_v(G)}}(r_{\phi}(e)-1)=\frac{1}{2}$, so we have the inequality $3 \leq 4+2+0.5=6.5$ by the second main theorem.
                      \item Let $\{a_1, a_q\}=\{v_{1}',v_{4}'\}\subset V(G')$. Then $|E\cap V(G)|=2$, and $$\frac{1}{2}\sum\limits_{v\in V(G)}\sum\limits_{e\in  {E_v(G)}}(r_{\phi}(e)-1)=0,$$ so we have the inequality $3 \leq 4+2=6$ by the second main theorem.
                      \item Let $\{a_1, a_q\}=\{v_{2}',v_{3}'\}\subset V(G')$. Then $|E\cap V(G)|=2$, and $$\frac{1}{2}\sum\limits_{v\in V(G)}\sum\limits_{e\in  {E_v(G)}}(r_{\phi}(e)-1)=1,$$ so we have the inequality $3 \leq 4+2+1=7$ by the second main theorem.
                    \end{itemize}
  \item Take $q=1$, \begin{itemize}
                       \item Let $ a_q\in V(G')$ take $v_{1}'$, or, $v_{4}'$, then $\frac{1}{2}\sum\limits_{v\in V(G)}\sum\limits_{e\in  {E_v(G)}}(r_{\phi}(e)-1)=0$ and $|E\cap V(G)|=1$, so we have the inequality $0 \leq 4+1=5$ by the second main theorem.
                       \item Let $ a_q\in V(G')$ take $v_{2}'$, or, $v_{3}'$, then $\frac{1}{2}\sum\limits_{v\in V(G)}\sum\limits_{e\in  {E_v(G)}}(r_{\phi}(e)-1)=\frac{1}{2}$ and $|E\cap V(G)|=1$,  so we have the inequality $0 \leq 4+1+0.5=5.5$ by the second main theorem.
                     \end{itemize}
\end{itemize}
\end{example}

\begin{figure}
  \centering
  \includegraphics[width=8cm]{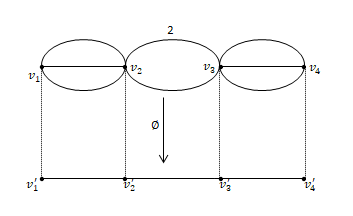}\\
  \caption{A pseudo-harmonic indexed morphism $\phi$ of degree three }\label{fig.4}
\end{figure}

\begin{example} The indexed morphism shown in Figure \ref{fig.5} is a harmonic morphism $\phi: (G, w)\to (G', w').$ Assume that all the indexes of vertical edges are $0$ and the indexes of horizontal edges are $1.$ All $M_{\phi}(v)$ and weights for vertexes are shown in the figure. Then we have:\\
$g=b_1(G)+\sum\limits_{v \in V(G)}w(v)=5,$ $g'=b_1(G')+\sum\limits_{v \in V(G')}w'(v')=5,$ and $\deg(\phi)=1.$ What's more, for each $v\in \{v_{1}, v_{2}, v_{4}, v_{5}\}$ we have $(w(v)-M_{\phi}(v)w'(v'))=0$ and $w(v_{3})-M_{\phi}(v_{3})w'({v'}_{3})=-1.$\par

\begin{itemize}
  \item Take $q=5$, let $\{a_1, \ldots, a_q\}=\{v_{1}', \ldots, v_{5}'\}=V(G').$ Then one can get that $|E\cap V(G)|=6$ and $\frac{1}{2}\sum\limits_{v\in V(G)}\sum\limits_{e\in  {E_v(G)}}(r_{\phi}(e)-1)=-2.$ By the second main theorem we get that $(5+5-1)\times 1\leq 5-1+6+1-2,$ which in fact is an equality. This means that the inequality of the second main theorem is sharp.
  \item Take $q=4$, \begin{itemize}
                      \item Let $\{a_1, \ldots, a_q\}=\{v_{1}',v_{3}',v_{4}', v_{5}'\}\subset V(G')$ (or can take $\{v_{2}',v_{3}',v_{4}', v_{5}'\}$, $\{v_{1}',v_{2}',v_{3}', v_{4}'\},$ $\{v_{1}',v_{2}',v_{3}', v_{5}'\}$). Then $\frac{1}{2}\sum\limits_{v\in V(G)}\sum\limits_{e\in  {E_v(G)}}(r_{\phi}(e)-1)=-2$, $|E\cap V(G)|=5$,  so we have the equality $8= 4+5+1-2=8$ by the second main theorem.
                      \item let $\{a_1, \ldots, a_q\}=\{v_{1}',v_{2}',v_{4}', v_{5}'\}\subset V(G')$. Then $|E\cap V(G)|=4$, and  $\frac{1}{2}\sum\limits_{v\in V(G)}\sum\limits_{e\in  {E_v(G)}}(r_{\phi}(e)-1)=0$ so we have the equality $8= 4+4=8$ by the second main theorem.
                    \end{itemize}
  \item Take $q=3$,\begin{itemize}
                     \item  Let $\{a_1, \ldots, a_q\}=\{v_{1}',v_{3}',v_{4}'\}\subset V(G')$ (or $=\{v_{2}',v_{3}',v_{4}'\}$, $\{v_{1}',v_{2}',v_{3}' \},$ $\{v_{1}',v_{3}', v_{5}'\}$, $\{v_{2}',v_{3}',v_{5}'\}$, $\{v_{3}',v_{4}',v_{5}' \},$). Then $\frac{1}{2}\sum\limits_{v\in V(G)}\sum\limits_{e\in  {E_v(G)}}(r_{\phi}(e)-1)=-2$ and  $|E\cap V(G)|=4$,  so we have the equality $7= 4+4+1-2=7$ by the second main theorem.
                     \item Let $\{a_1, \ldots, a_q\}=\{v_{1}',v_{4}',v_{5}'\}\subset V(G')$ (or $=\{v_{2}',v_{4}',v_{5}'\}$, $\{v_{1}',v_{2}',v_{4}' \},$ $\{v_{1}',v_{2}', v_{5}'\}$). Then $|E\cap V(G)|=3$,  $\frac{1}{2}\sum\limits_{v\in V(G)}\sum\limits_{e\in  {E_v(G)}}(r_{\phi}(e)-1)=0$ so we have the equality $7= 4+3=7$ by the second main theorem.
                   \end{itemize}
  \item Take $q=2$,\begin{itemize}
                     \item  Let $\{a_1, a_q\}=\{v_{1}',v_{3}'\}\subset V(G^{'})$ (or $=\{v_{2}',v_{3}'\}$, $\{v_{3}',v_{4}'\}$,  $\{v_{3}',v_{5}'\}$). Then $|E\cap V(G)|=3$,  $\frac{1}{2}\sum\limits_{v\in V(G)}\sum\limits_{e\in  {E_v(G)}}(r_{\phi}(e)-1)=-2$ so we have the equality $6= 4+3+1-2=6$ by the second main theorem.
                     \item Let $\{a_1, a_q\}=\{v_{1}',v_{2}'\}\subset V(G')$ (or $=\{v_{1}',v_{4}'\}$, $\{v_{1}',v_{5}'\}$,  $\{v_{2}',v_{4}'\}$,  $\{v_{2}',v_{5}'\}$,  $\{v_{4}',v_{5}'\}$). Then $|E\cap V(G)|=2$,  $\frac{1}{2}\sum\limits_{v\in V(G)}\sum\limits_{e\in  {E_v(G)}}(r_{\phi}(e)-1)=0$ so we have the equality $6= 4+2=6$ by the second main theorem.
                   \end{itemize}
  \item Take $q=1$, \begin{itemize}
                      \item Let $a_q\in V(G')$ take $v_{1}'$, or, $v_{2}'$, or, $v_{4}'$, or, $v_{5}'$, then $|E\cap V(G)|=1$,  $\frac{1}{2}\sum\limits_{v\in V(G)}\sum\limits_{e\in  {E_v(G)}}(r_{\phi}(e)-1)=0$ so we have the equality $5= 4+1=5$ by the second main theorem.
                      \item Let $a_q\in V(G')$ take $v_{3}'$, then $|E\cap V(G)|=2$,  $\frac{1}{2}\sum\limits_{v\in V(G)}\sum\limits_{e\in  {E_v(G)}}(r_{\phi}(e)-1)=-2$ so we have the equality $5= 4+2+1-2=5$ by the second main theorem.
                    \end{itemize}
\end{itemize}
\end{example}

\begin{figure}
  \centering
  \includegraphics[width=8cm]{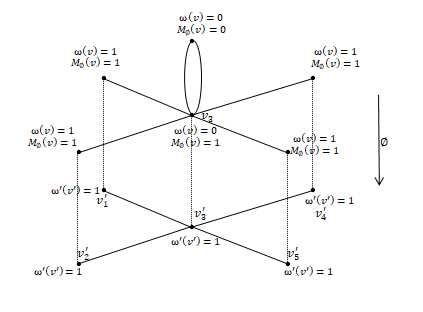}\\
  \caption{A harmonic morphism $\phi$ of degree one}\label{fig.5}
\end{figure}

\section{Second main theorem on metric graphs}
\subsection{Harmonic morphism between metric graphs}\par

It is known in the subsection 2.3 that if the model of a metric graph has loops, we can replace an additional vertex at the midpoint of each loop edge so that it becomes loopless model for the metric graph. So we will only assume that all models of metric graphs are loopless in this section. Let's recall some definitions as in \cite{Cha-2013}. \par

Let $(G,\ell)$ and $(G',{\ell}')$ be loopless models for two metric graphs $\Gamma$ and $\Gamma'$, respectively. A morphism of loopless models $\phi: (G,\ell) \to (G',{\ell}')$ is a map of sets $$ \phi: V(G)\cup E(G)\to V(G')\cup E(G')$$ such that \par
(i)  $\phi(V(G))\subseteq V(G')$;\par
(ii) if $e=xy$ is an edge of $G$ and $\phi(e)\in V(G')$ then $\phi(x)=\phi(e)=\phi(y)$;\par
(iii) if $e=xy$ is an edge of $G$ and $\phi(e)\in E(G')$ then  $\phi(e)$ is an edge between $\phi(x)$ and $\phi(y)$;\par
(iv) if $\phi(e)={e}'$ then ${\ell}'(e')/{\ell}(e)$ is an integer.\\
An edge $e\in E(G)$ is called horizontal if $\phi(e)\in E(G^{'})$ and vertical if $\phi(e)\in V(G^{'}).$ Denote $$U_{\phi}(e)={\ell}'(e')/{\ell}(e) \in \mathbb{Z} $$ to be the slope of this linear map.\par

\begin{definition}\cite{Cha-2013} \label{def5.1} A harmonic morphism between metric graphs $\Gamma$ and $\Gamma^{'}$ is viewed as the morphism $\phi: (G, \ell)\to (G',{\ell}')$ of loopless models, for some choice of models $(G,\ell)$ and $(G',{\ell}'),$ which satisfies that for every $v\in V(G)$, the nonnegative integer $$M_{\phi}(v)=\sum_{e\in E(G)\atop v\in e,\phi(e)=e'}U_{\phi}(e)$$ is the same for all edges $e'\in E(G')$ that are incident to the vertex $\phi(v)$. The number $M_{\phi}(v)$ is called the horizontal multiplicity of $\phi$ at $x.$\end{definition}

\begin{definition}\cite{Cha-2013} \label{def5.2} The degree of a harmonic morphism $\phi$ is defined to be $$\deg(\phi)=\sum_{e\in E(G)\atop \phi(e)=e'}U_{\phi}(e)$$ for any $e'\in E(G')$. If $G'$ has no edges, then set $\deg(\phi)=0.$
\end{definition}

It is known that the number $\deg(\phi)$ does not depend on the choice of $e'.$ If $M_{\phi}(v)\geq 1 $ for all $v\in V(G)$, then $\phi$ is said to be nondegenerate.

\begin{proposition} \label{pro5.3}For any vertex $v'\in V(G')$, we have $$\deg(\phi)=\sum_{v\in V(G)\atop \phi(v)=v'}M_{\phi}(v).$$\end{proposition}

\begin{proof}For any vertex $v'\in V(G')$, there exists one edge $e'\in E(G')$ such that $v'\in e'.$ Then
\begin{equation*}\deg(\phi)=\sum_{e\in E(G)\atop \phi(e)=e'}U_{\phi}(e)=\sum_{v\in V(G)\atop \phi(v)=v'}\sum_{e\in E(G)\atop v\in e,\phi(e)=e'}U_{\phi}(e)=\sum_{v\in V(G)\atop \phi(v)=v'}M_{\phi}(v).
\end{equation*}
\end{proof}

\begin{definition}\label{def5.4} Let $\phi: \Gamma\to {\Gamma}'$ be a harmonic morphism with two loopless models $(G, \ell)$ and $(G', \ell')$ respectively, the pullback map on divisor ${\phi}^*: Div(\Gamma')\to Div(\Gamma)$ is defined as follows: given $D'\in Div(\Gamma')$, let $$({\phi}^*(D'))(v)=M_{{\phi}}(v)\cdot D'({\phi}(v))$$ for all $v\in V(G).$
\end{definition}

\begin{lemma} \label{lem5.5} Let $\phi: \Gamma\to \Gamma'$ be a harmonic morphism. Then for a canonical divisor $K_{\Gamma'}\in Div(\Gamma')$,  we have $$\deg({\phi}^*(K_{\Gamma'}))=\deg({\phi})\deg(K_{\Gamma'}).$$ \end{lemma}

\begin{proof}
By Definition \ref{def5.4} we get that $${\phi}^*(K_{\Gamma'})=\sum_{v\in {V(G)}} M_{{\phi}}(v)\cdot K_{\Gamma'}({\phi}(v))(v)=\sum_{v'\in {V(G')}}\sum_{v\in V(G)\atop \phi(v)=v'} M_{{\phi}}(v)\cdot K_{\Gamma'}((v'))(v).$$ And $$\deg({\phi}^*(K_{\Gamma'}))=\sum_{v'\in {V(G')}}\sum_{v\in V(G)\atop \phi(v)=v'} M_{{\phi}}(v)\cdot K_{\Gamma'}((v')).$$
 Hence , by Proposition \ref{pro5.3}, we obtain the lemma.
\end{proof}

\subsection{Riemann-Hurwitz theorem for metric graphs}\par
We now prove the Riemann-Hurwitz theorem for metric graphs.\par

\begin{theorem} \label{the5.6} Let $\phi$ be a harmonic morphism between two metric graphs $\Gamma, \Gamma'$ with two loopless models $(G, \ell)$ and $(G', \ell'),$ respectively. Then\\

(i). the canonical divisors on $\Gamma$ and $\Gamma'$ are related by the formula $$K_\Gamma={\phi}^*K_{\Gamma'}+R_{\phi},$$ where $$R_{\phi}=\sum_{v\in V(G)}\left(2(M_{\phi}(v)-1)-\sum_{e\in E_v(G)}(U_{\phi}(e)-1)\right)(v).$$\\

(ii). $$2g-2=\deg(\phi)(2g'-2)+\sum_{v\in V(G)}(2(M_{\phi}(v)-1)-\sum_{e\in E_v(G)}(U_{\phi}(e)-1)),$$ where $g$ and $g'$ are genus of $\Gamma$ and $\Gamma^{'}$ respectively.
\end{theorem}

\begin{proof}
For every $v\in V(G)$, we have $K_{\Gamma}(v)=val(v)-2$. Then, writing $v'={\phi}(v)$, by Definition \ref{def5.4}, we have
\begin{eqnarray*}
K_\Gamma(v)-{\phi}^*K_{\Gamma'}(v) &=& val(v)-2-M_{\phi}(v)(val(v')-2)\nonumber\\
&=& 2(M_{\phi}(v)-1)+val(v)-M_{\phi}(v)val(v').
\end{eqnarray*}
On the other hand, by Definition \ref{def5.1}, we have $$\sum_{e\in E_v(G)}(U_{\phi}(e)-1)=\sum_{e\in E_v(G)}U_{\phi}(e)-val(v)=M_{\phi}(v)val(v')-val(v)$$ for the $v\in V(G).$ The two above identities imply $$K_\Gamma(v)-{\phi}^*K_{\Gamma'}(v)=R_{\phi}(v),$$ so the conclusion (i) is proved.\par

Notice that the genus of a metric graph is independent of the choice of its models. The conclusion (ii) follows immediately from Lemma \ref{lem5.5} upon computing the degrees of the divisors on both sides of the above formula.
\end{proof}

\subsection{Second main theorem on metric graphs} Now, the Riemann-Hurwitz theorem for metric graphs gives the second main theorem for harmonic morphisms on metric graphs as follows.\par

\begin{theorem} Let $\Gamma$, $\Gamma'$ be two metric graphs with genus $g$ and $g'$, respectively, and let ${\phi}:\Gamma\to \Gamma'$ be a harmonic morphism for some choice of loopless models $(G, \ell)$ and $(G', \ell').$  Suppose $a_1, \ldots, a_q\in V(G')$ be distinct vertices, let $E={\phi}^{-1}(\{a_1, \ldots, a_q\}).$ Then we have $$(q+g'-1)\deg(\phi)\leq g-1+|E\cap V(G)|+ \frac{1}{2}\sum\limits_{v\in V(G)}\sum\limits_{e\in {E_v(G)}}(U_{\phi}(e)-1),$$ where $|E\cap V(G)|$ is the cardinality of $E\cap V(G).$
\end{theorem}

\begin{proof}
Set $$r_{\phi}(E):=\sum_{v\in E\cap V(G)}(M_{\phi}(v)-1),$$ and $$r_{\phi}(\Gamma):=\sum_{v\in V(G)}(M_{\phi}(v)-1).$$ It is obvious that $$r_{\phi}(E)\leq r_{\phi}(\Gamma).$$ From the definition of the degree of the harmonic morphism $\phi:(G, \ell)\to (G', \ell')$ of loopless models, the horizontal multiplicity of $\phi$ at $v$ and Proposition \ref{pro5.3} we get that $$\sum_{v\in E\cap V(G)\atop a_j=\phi(v)}M_{\phi}(v)=\sum_{v\in V(G)\atop a_j=\phi(v)}M_{\phi}(v)=\deg(\phi)$$ holds for each $ a_j\in \{a_1, \ldots, a_q\}.$ Then we have  \begin{eqnarray*}r_{\phi}(E)&=&\sum_{v\in E\cap V(G)}(M_{\phi}(v)-1)\\\nonumber&=&(\sum\limits_{j=1}^q \sum\limits_{v\in E\cap V(G)\atop a_j=\phi(v)}M_{\phi}(v))-|E\cap V(G)|\\\nonumber&=&q\deg(\phi)-|E\cap V(G)|.\end{eqnarray*}  On the other hand, by the Riemann-Hurwitz theorem for metric graphs (Theorem \ref{the5.6}), we have
$$r_{\phi}(\Gamma)=g-1-(g'-1)\deg(\phi)+\frac{1}{2}\sum\limits_{v\in V(G)}\sum_{e\in E_{v}(G)}(U_{\phi}(e)-1).$$ Hence, we get the following inequality $$(q+g'-1)\deg(\phi)\leq g-1+|E\cap V(G)|+\frac{1}{2}\sum\limits_{v\in V(G)}\sum_{e\in E_{v}(G)}(U_{\phi}(e)-1).$$
\end{proof}
We give an example that satisfies the second main theorem for metric graphs.\par
\begin{figure}
  \centering
  \includegraphics[width=8cm]{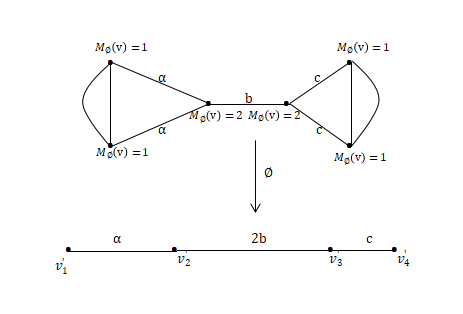}\\
  \caption{A harmonic morphism of degree two}\label{fig.6}
\end{figure}

\begin{example} In the following Figure \ref{fig.6} \cite[Fig.1]{Cha-2013}, we have a harmonic morphism $\phi$ on metric graphs, labeled next to the corresponding vertices. Here, $a$, $b$, and $c$ are positive real numbers. It is known by calculation that:\\
$g=|E(G)|-|V(G)|+1=4,$ $g'=|E(G^{'})|-|V(G^{'})|+1=0,$ and $\deg(\phi)=2.$ \par
\begin{itemize}
  \item Take $q=4$ and let $\{a_1, \ldots, a_q\}=\{v_{1}', \ldots, v_{4}'\}=V(G').$ Then one can get that $|E\cap V(G)|=6$ and $\frac{1}{2}\sum\limits_{v\in V(G)}\sum\limits_{e\in E_{v}(G)}(U_{\phi}(e)-1)=-3.$ By the second main theorem we get that $(4-1)\times 2\leq 3+6-3=6,$ which in fact is an equality. This means that the inequality of the second main theorem is sharp.
  \item Take $q=3$. \begin{itemize}
                      \item Let $\{a_1, \ldots, a_q\}=\{v_{1}',v_{2}',v_{3}'\}\subset V(G')$ (or $=\{v_{2}',v_{3}',v_{4}'\}$). Then $|E\cap V(G)|=4$, and $\frac{1}{2}\sum\limits_{v\in V(G)}\sum\limits_{e\in E_{v}(G)}(U_{\phi}(e)-1)=-1$, so we have the inequality $4 \leq 3+4-1=6$ by the second main theorem.
                      \item Let $\{a_1, \ldots, a_q\}=\{v_{1}',v_{2}',v_{4}'\}\subset V(G')$ (or $=\{v_{1}',v_{3}',v_{4}'\}$). Then$|E\cap V(G)|=5$, and $\frac{1}{2}\sum\limits_{v\in V(G)}\sum\limits_{e\in E_{v}(G)}(U_{\phi}(e)-1)=-3.5$, so we have the inequality $4 \leq 3+5-3.5=4.5$ by the second main theorem.
                    \end{itemize}
  \item  Take $q=2$, \begin{itemize}
                       \item  Let $\{a_1, a_q\}=\{v_{1}',v_{2}'\}\subset V(G')$ (or $=\{v_{1}',v_{3}'\}, \{v_{2}',v_{4}'\}, \{v_{3}',v_{4}'\}$). Then $|E\cap V(G)|=3$, and  $\frac{1}{2}\sum\limits_{v\in V(G)}\sum\limits_{e\in E_{v}(G)}(U_{\phi}(e)-1)=-1.5$, so we have the inequality $2 \leq 3+3-1.5=4.5$ by the second main theorem.
                       \item Let $\{a_1, a_q\}=\{v_{1}',v_{4}'\}\subset V(G')$. Then $\frac{1}{2}\sum\limits_{v\in V(G)}\sum\limits_{e\in E_{v}(G)}(U_{\phi}(e)-1)=-4$ and $|E\cap V(G)|=4$, so we have the inequality $2 \leq 3+4-4=3$ by the second main theorem.
                       \item Let $\{a_1, a_q\}=\{v_{2}',v_{3}'\}\subset V(G')$. Then $\frac{1}{2}\sum\limits_{v\in V(G)}\sum\limits_{e\in E_{v}(G)}(U_{\phi}(e)-1)=1$  and  $|E\cap V(G)|=2$, so we have the inequality $2 \leq 3+2+1=6$ by the second main theorem.
                     \end{itemize}
  \item Take $q=1$, \begin{itemize}
                      \item Let $a_q\in V(G')$ take $v_{1}'$, or, $v_{4}'$. Then $\frac{1}{2}\sum\limits_{v\in V(G)}\sum\limits_{e\in E_{v}(G)}(U_{\phi}(e)-1)=-2$ and $|E\cap V(G)|=2$, so we have the inequality $0 \leq 3+2-2=3$ by the second main theorem.
                      \item Let $a_q\in V(G')$ take $v_{2}'$, or, $v_{3}'$. Then  $\frac{1}{2}\sum\limits_{v\in V(G)}\sum\limits_{e\in E_{v}(G)}(U_{\phi}(e)-1)=0.5$ and $|E\cap V(G)|=1$, so we have the inequality $0 \leq 3+1+0.5=4.5$ by the second main theorem.
                    \end{itemize}
\end{itemize}
\end{example}

\section{Second main theorem for vertex-weighted metric graphs}
 \subsection{Pseudo-harmonic morphisms between vertex-weighted metric graphs}
 From Subsection 2.4., for a vertex-weighted metric graph $(\Gamma, w),$ we can associate it the \emph{pseudo-metric graph} $(G^w, \ell^w)$  and then consider a \emph{pure metric graph} $\Gamma_{\varepsilon}^w.$ Conversely, when we deal with a vertex-weighted metric graph with loops, we can add the vertex-weighted number corresponding to the number of loops at that point. So we will only suppose that all vertex-weighted metric graphs are loopless in this section.\par

Let $(\Gamma,w)=(G, w, \ell)$ and $({\Gamma}', w')=(G', w', {\ell}')$ be loopless vertex-weighted metric graphs. A morphism between vertex-weighted metric graphs  $(\Gamma,w)$ and $({\Gamma}',w')$ is defined as a morphism of loopless models $\phi: (G,\ell) \to (G',{\ell}')$ for metric graphs $\Gamma$ and ${\Gamma}'$. We now introduce the definition of  pseudo-harmonic morphism between vertex-weighted metric graphs.\par

\begin{definition}\label{def6.1} Let $(\Gamma,w)=(G,w,\ell)$ and $({\Gamma}',w')=(G',w',{\ell}')$ be loopless vertex-weighted metric graphs. Suppose that $(G,\ell)$ and $(G',{\ell}')$ are loopless models for $\Gamma$ and $\Gamma'$, respectively.\par

(i) For every $e\in E(G)$, there is always a non-negative integer, the slope of $\phi$ at $e$, also written $U_{\phi}(e)$, where $$U_{\phi}(e)={{\ell}'(e')/\ell(e)}\in \mathbb{Z}.$$ We can easily find out $U_{\phi}(e)=0$ if and only if $\phi(e)$ is a point. If for every $e\in E(G), U_{\phi}(e)\leq 1$, then the morphism is simple.\par

(ii) A morphism is pseudo-harmonic if for every $v\in V(G)$, there exists a nonnegative integer $$M_{\phi}(v)=\sum_{e\in E(G)\atop v\in e,\phi(e)=e'}U_{\phi}(e)$$ is the same for all edges $e'\in E(G')$ that are incident to the vertex $\phi(v)$.

(iii)  A harmonic morphism is non-degenerate if $M_{\phi}(v) \geq 1$, for every $v\in V(G).$\par
(iv) A pseudo-harmonic morphism is harmonic if for every $v\in V(G)$ we have, writing ${v}'= {\phi}(v)$,  $$\sum_{e\in {E(G)}}(U_{\phi}(e)-1)\leq 2(M_{\phi}(v)-1+w(v)-M_{\phi}(v)w'(v')).$$
\end{definition}

\begin{remark} Suppose that $\phi$ contracts a leaf-edge $e$ whose leaf-vertex $v$ has weight-zero. Then $U_{\phi}(e)=M_{\phi}(v)=0$ and we can known that $(iv)$ is not satisfied on the point $v.$ So, loosely speaking, a harmonic morphism can't contract weight-zero leaves.
\end{remark}
\begin{remark} [Relation with harmonic morphisms of metric graphs $\Gamma$ and ${\Gamma}'$] For morphisms of vertex-weightless metric graphs our definition of harmonic morphism between vertex-weighted metric graphs $(\Gamma, w)=(G, w, \ell)$ and $({\Gamma}', w')=(G', w', {\ell}')$ coincides with the harmonic morphisms of metric graphs $\Gamma$ and ${\Gamma}'$ which contract no leaves.
\end{remark}

\begin{remark}[Relation with pseudo-harmonic indexed (resp. harmonic) morphism of vertex-weighted graphs] For simple morphisms our definition of pseudo-harmonic (resp.harmonic) morphism between vertex-weighted metric graphs $(\Gamma, w)$ and $({\Gamma}', w')$ coincides with the pseudo-harmonic indexed (resp. harmonic) morphism between simple vertex-weighted graphs $(G, w)$ and $(G',w')$. One notable difference is that in \cite{Cap-2014}, only the combinatorial type of the metric graphs are fixed; the choice of positive indices in a pseudo-harmonic indexed morphism determines the length of the edges in the source graph once the edge lengths in the target are fixed.
\end{remark}

\begin{remark} [Relation with harmonic morphisms of finite graphs \cite{BN-2009}] For simple morphisms of vertex-weightless metric graphs, the above definition of harmonic morphism coincides with the one given in Section 3 (\cite{BN-2009}) for morphisms which contract no leaves.\end{remark}

We now define the degree of $\phi$ as follows.\par

\begin{definition}
 If $\phi: (\Gamma,w)\to (\Gamma',w')$ be a pseudo-harmonic morphism. Then for every $e'\in E(G')$, the degree of $\phi$ is defined as follows $$\deg(\phi)=\sum_{e\in E(G)\atop \phi(e)=e'}U_{\phi}(e).$$ If $G'$ has no edges, then set $\deg(\phi)=0.$
\end{definition}

By the (ii) of Definition \ref{def6.1}, we know that the $\deg(\phi)$ does not depend on the choice of $e'.$ With the same proof, we also have the same result as Proposition \ref{pro5.3}.\par

\begin{lemma}\label{lem6.7} Let $\phi: (\Gamma,w)\to (\Gamma',w')$ be a pseudo-harmonic morphism with two loopless models $(G, \ell)$ and $(G', \ell')$ respectively. For any vertex $v'\in V(G')$, we have $$\deg(\phi)=\sum_{v\in V(G)\atop \phi(v)=v'}M_{\phi}(v).$$\end{lemma}

\begin{definition}\label{def6.8} Let $\phi: (\Gamma,w)\to (\Gamma',w')$ be a pseudo-harmonic morphism with two loopless models $(G, \ell)$ and $(G', \ell')$ respectively, the pullback map on divisor ${\phi}^*: Div({\Gamma}',w')\to Div({\Gamma},w)$ is defined as follows: given $D'\in Div({\Gamma}',w')$, $${\phi}^*(D')(v)=M_{{\phi}}(v)\cdot D'({\phi}(v))$$ for all vertex $v\in V(G).$
\end{definition}

\begin{lemma} \label{lem6.9} Let $\phi: (\Gamma,w)\to (\Gamma',w')$ be a pseudo-harmonic morphism with two loopless models $(G, \ell)$ and $(G', \ell')$ respectively. Then for a canonical divisor $K_{(\Gamma', w')},$ we have $$\deg({\phi}^*(K_{(\Gamma^{'}, w')}))=\deg({\phi})\deg(K_{(\Gamma^{'}, w')}).$$ \end{lemma}

\begin{proof}
By Definition \ref{def6.8}, we get that \begin{eqnarray*}&&{\phi}^*(K_{(\Gamma', w')})\\&=&\sum_{v\in {V(G)}} M_{{\phi}}(v)\cdot K_{(\Gamma', w')}({\phi}(v))(v)\\&=&\sum_{v'\in {V(G')}}\sum_{v\in V(G)\atop \phi(v)=v'} M_{{\phi}}(v)\cdot K_{(\Gamma', w')}((v'))(v).\end{eqnarray*} And $$\deg({\phi}^*(K_{(\Gamma', w')}))=\sum_{v'\in {V(G')}}\sum_{v\in V(G)\atop \phi(v)=v'} M_{{\phi}}(v)\cdot K_{(\Gamma', w')}((v')).$$
 Hence , by Lemma \ref{lem6.7}, we obtain the lemma.

\end{proof}

\subsection{Riemann-Hurwitz theorem for vertex-weighted metric graphs}\par
\begin{theorem} \label{the6.10} Let $(\Gamma,w)=(G,w,\ell)$ and $({\Gamma}',w')=(G',w',{\ell}')$ be loopless vertex-weighted metric graphs. Let $\phi: (\Gamma,w)\to (\Gamma',w')$ be a pseudo-harmonic morphism with two loopless models $(G, \ell)$ and $(G', \ell'),$ respectively. Then\\

(i). the canonical divisors on $(\Gamma,w)$ and $(\Gamma',w')$ are related by the formula $$K_{(\Gamma,w)}={{\phi}}^*K_{(\Gamma',w')}+R_{{\phi}},$$ where \begin{eqnarray*}
R_{\phi}:=\sum_{v\in V(G)}\left(2(M_{\phi}(v)-1+w(v)-M_{\phi}(v)w'(v'))-\sum\limits_{e\in {E_v(G)}}(U_{\phi}(e)-1)\right)(v),
\end{eqnarray*}
(ii). \begin{eqnarray*}
2g-2 &=& \deg(\phi)(2g'-2)+\sum_{v\in V(G)}2(M_{\phi}(v)-1+w(v)-M_{\phi}(v)w'(v'))\nonumber\\
&\;&- \sum\limits_{v\in V(G)}\sum\limits_{e\in  {E_v(G)}}(U_{\phi}(e)-1),
\end{eqnarray*} where $g$ and $g'$ are genus of $(\Gamma,w)$ and $(\Gamma',w')$ respectively, and $v'=\phi(v).$
\end{theorem}

\begin{proof}
For every $v\in V(G)$, we have $K_{(\Gamma,w)}(v)=val(v)-2+2w(v)$. Then, writing $v'={\phi}(v)$, by Definition \ref{def6.8}, we have
\begin{eqnarray*}&&
K_{(\Gamma,w)}(v)-{\phi}^*K_{({\Gamma}',w')}(v) \\&=& val(v)-2+2w(v)-M_{\phi}(v)(val(v')-2+2w'(v'))\nonumber\\
&=& 2(M_{\phi}(v)-1+w(v)-M_{\phi}(v)w'(v'))+val(v)-M_{\phi}(v)val(v').
\end{eqnarray*}
On the other hand, by Definition \ref{def6.1}, we have  $$\sum_{e\in E_v(G)}(U_{\phi}(e)-1)=\sum_{e\in E_v(G)}U_{\phi}(e)-val(v)=M_{\phi}(v)val(v')-val(v)$$ for the $v\in V(G).$ The two above identities imply  $$K_{(\Gamma,w)}(v)-{\phi}^*K_{({\Gamma}',w')}(v)=R_{\phi}(v),$$ so the conclusion (i) is proved.\par
Notice that the genus of a vertex-weighted metric graph is independent of the choice of its models. The conclusion (ii) follows immediately from Lemma \ref{lem6.9} upon computing the degrees of the divisors on both sides of the above formula.
\end{proof}

\subsection{Second main theorem on vertex-weighted  metric graphs} In this subsection, we prove the second main theorem on vertex-weighted metric graphs as follows.\par

\begin{theorem} Let $(\Gamma,w)=(G,w,\ell)$ and $({\Gamma}',w')=(G',w',{\ell}')$ be loopless vertex-weighted metric graphs. Suppose that $\phi: (G,w,\ell) \to (G',w',{\ell}')$ are a pseudo-harmonic morphism with two loopless models $(G, \ell)$ and $(G^{'}, \ell^{'})$, and have genus $g$ and $g'$, respectively. Assume that $a_1, \ldots, a_q$ are distinct vertices in $V(G').$ Set $E={\phi}^{-1}(\{a_1, \ldots, a_q\}).$ Then we have
\begin{eqnarray*}
(q+g'-1)\deg(\phi) &\leq& g-1+|E\cap V(G)|-\sum_{v\in V(G)}(w(v)-M_{\phi}(v)w'(v'))\nonumber\\
&\;&+\frac{1}{2}\sum\limits_{v\in V(G)}\sum\limits_{e\in {E(G)}}(U_{\phi}(e)-1).
\end{eqnarray*} where $v'=\phi(v)$ and $|E\cap V(G)|$ is the cardinality of $E\cap V(G).$
\end{theorem}

\begin{proof}
Set $$r_{\phi}(E):=\sum_{v\in E\cap V(G)}(M_{\phi}(v)-1),$$ and $$r_{\phi}(\Gamma,w):=\sum_{v\in V(G)}(M_{\phi}(v)-1).$$ It is obvious that $$r_{\phi}(E)\leq r_{\phi}(\Gamma,w).$$ From the definition of the degree of the pseudo-harmonic morphism $\phi: (\Gamma,w)\to (\Gamma',w')$ and Lemma \ref{lem6.7} we get that $$\sum_{v\in E\cap V(G)\atop a_j=\phi(v)}M_{\phi}(v)=\sum_{v\in V(G)\atop a_j=\phi(v)}M_{\phi}(v)=\deg(\phi)$$ holds for each $ a_j\in \{a_1, \ldots, a_q\}.$ Then we have  \begin{eqnarray*}r_{\phi}(E)&=&\sum_{v\in E\cap V(G)}(M_{\phi}(v)-1)\\\nonumber&=&(\sum\limits_{j=1}^q \sum\limits_{v\in E\cap V(G)\atop a_j=\phi(v)}M_{\phi}(v))-|E\cap V(G)|\\\nonumber&=&q\deg(\phi)-|E\cap V(G)|.\end{eqnarray*}  On the other hand, by the Riemann-Hurwitz theorem for vertex-weighted metric graphs (Theorem \ref{the6.10}), we have
\begin{eqnarray*}r_{\phi}(\Gamma,w)&=&g-1-(g'-1)\deg(\phi)-\sum_{v\in V(G)}(w(v)-M_{\phi}(v)w'(v'))\\&&+\frac{1}{2}\sum\limits_{v\in V(G)}\sum_{e\in E(G)}(U_{\phi}(e)-1).\end{eqnarray*} Hence, we get the following inequality
\begin{eqnarray*}
(q+g'-1)\deg(\phi) &\leq& g-1+|E\cap V(G)|-\sum_{v\in V(G)}(w(v)-M_{\phi}(v)w'(v'))\nonumber\\
&\;&+\frac{1}{2}\sum\limits_{v\in V(G)}\sum\limits_{e\in {E_v(G)}}(U_{\phi}(e)-1).
\end{eqnarray*}
\end{proof}

We give an example that satisfies the second main theorem for vertex-weighted metric graphs.\par

\begin{example} Following Figure \ref{fig.7}, a harmonic morphism for suitable choices of lengths satisfies that $U_{\phi}(e)$ is zero for each vertical edge  and is equal to one for every horizontal edge. All weights at vertexes are labeled next to the corresponding vertexes in the figure. Then we can get that \\
$g=b_1(\Gamma)+\sum\limits_{v \in V(G)}w(v)=4,$ $g'=b_1({\Gamma}')+\sum\limits_{v \in V(G')}w'(v')=1,$ and $\deg(\phi)=2,$\par

\begin{itemize}
  \item Take $q=3$ and let $\{a_1, \ldots, a_q\}=\{v_{1}', \ldots, v_{3}'\}=V(G^{'}).$ Then one can get that $|E\cap V(G)|=6,$ $\frac{1}{2}\sum\limits_{v\in V(G)}\sum\limits_{e\in E(G)}(U_{\phi}(e)-1)=-2,$ and $\sum\limits_{v\in V(G)}(w(v)-M_{\phi}(v)w'(v'))=1,$ and by the second main theorem we get that $(3+1-1)\times 2\leq 3+6-1+2=6,$ which in fact is an equality. This means that the inequality of the second main theorem is sharp.
  \item  Take $q=2$, \begin{itemize}
                       \item Let $\{a_1, a_q\}=\{v_{1}',v_{2}'\}\subset V(G')$ (or $=\{v_{2}',v_{3}'\}$). Then $|E\cap V(G)|=4,$ $\sum\limits_{v\in V(G)}(w(v)-M_{\phi}(v)w'(v'))=1,$ and $\frac{1}{2}\sum\limits_{v\in V(G)}\sum\limits_{e\in E_{v}(G)}(U_{\phi}(e)-1)=-2,$  so we have the equality $4 \leq 3+4-1-2=4$ by the second main theorem.
                       \item Let $\{a_1, a_q\}=\{v_{1}',v_{3}'\}\subset V(G')$. Then $|E\cap V(G)|=4$ and $\sum\limits_{v\in V(G)}(w(v)-M_{\phi}(v)w'(v'))=0,$ $\frac{1}{2}\sum\limits_{v\in V(G)}\sum\limits_{e\in E_{v}(G)}(U_{\phi}(e)-1)=0,$ so we have the inequality $4 \leq 3+4=7$ by the second main theorem.
                     \end{itemize}
  \item Take $q=1$, \begin{itemize}
                      \item Let $a_q\in V(G')$ take $v_{1}'$, or, $v_{3}'$. Then $|E\cap V(G)|=2$ and $\sum\limits_{v\in V(G)}(w(v)-M_{\phi}(v)w'(v'))=0,$ $\frac{1}{2}\sum\limits_{v\in V(G)}\sum\limits_{e\in E_{v}(G)}(U_{\phi}(e)-1)=0,$ so we have the inequality $2 \leq 3+2=5$ by the second main theorem.
                      \item Let $a_q\in V(G')$ take $v_{2}'$. Then $|E\cap V(G)|=2$, and $\sum\limits_{v\in V(G)}(w(v)-M_{\phi}(v)w'(v'))=1,$ $\frac{1}{2}\sum\limits_{v\in V(G)}\sum\limits_{e\in E_{v}(G)}(U_{\phi}(e)-1)=-2$, so we have the equality $2 \leq 3+2-1-2=2$ by the second main theorem.
                    \end{itemize}
\end{itemize}
\end{example}
\begin{figure}
  \centering
  \includegraphics[width=8cm]{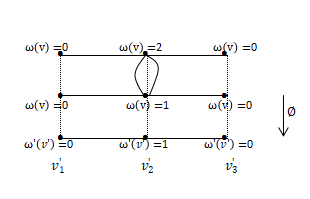}\\
  \caption{A harmonic morphism of degree two}\label{fig.7}
\end{figure}

\section{Second main theorem for metrized complexes of algebraic curves}

\subsection{Harmonic morphism between metrized complexes of algebraic curves}\par
Let $\mathfrak{C}$, $\mathfrak{C'}$ be metrized complexes of algebraic curves on an algebraically closed field $\kappa$ whose underlying vertex-weighted metric graphs are $(\Gamma, w)$ and $(\Gamma^{'}, w'),$ respectively. Without loss of generality, we assume that $(G, \ell)$ and $(G', {\ell}')$ are loopless models for metric graphs $\Gamma$ and ${\Gamma}',$ respectively, and the associated curves of $\mathfrak{C}$ and $\mathfrak{C'}$ are $\{\mathcal{C}_v\}_{v\in G}$ and  $\{\mathcal{C}_{v'}^{'}\}_{v'\in G'}$  respectively. We use the definition of harmonic morphism between metrized complexes coming from \cite{ABBR1-2015, ABBR2-2015} and therein references. In which we may let the morphism $\phi_{v}$ be a nonconstant holomorphic map between algebraic curves whenever $\kappa=\mathbb{C}.$\par

\begin{definition}\label{D7.1} A harmonic morphism $\varphi=(\phi, \{\phi_v\}_{v\in V(G)})$ between metrized complexes $\mathfrak{C}$, $\mathfrak{C'}$ consists of a harmonic morphism ${\phi}: (\Gamma, w) \to (\Gamma', w')$ of vertex-weighted metric graphs,  and for every vertex $v\in V(G)$ of $\Gamma$ with $M_{\phi}(v)>0$ a finite morphism of algebraic curves ${\phi}_v: {\mathcal{C}}_v\to {{\mathcal{C}'}_{\phi(v)}}$, satisfying the following compatibility conditions: \par

(i) For every vertex $v\in V(G)$ and every edge $e\in E_v(G)$ with $U_{\phi}(e)>0$, have ${\phi_v}(red_v(e))=red_{\phi(v)}(U_{\phi(v)}(e)),$ where $U_{\phi (v)}$ is a map induced by $\phi$, $$U_{\phi(v)}: \{e\in E_v(G): U_{\phi}(e) \neq 0\}\to E_{v'}(G').$$\par

(ii) For every vertex $v\in V(G)$ and every edge $e\in E_v(G)$ with $U_{\phi}(e)>0$, the ramification index of $\phi_v$ at the marked point corresponding to the edge $e$ is equal to $U_{\phi}(e)$.\par
(iii) For every vertex $v\in V(G)$ with $M_{\phi}(v)>0$, every $e' \in E_{\phi(v)}(G')$, and every point $x\in {\phi}^{-1}_{v}(red_{\phi(v)}(e')) \subset {\mathcal{C}}_v,$ there exists $e\in E_v(G)$ such that $red_v(e)=x$.\par
(iv) For every vertex $v\in V(G)$ with $M_{\phi}(v)>0$ we have $M_{\phi}(v)$ = $\deg(\phi_v).$\end{definition}\par

Next we give the definition of the degree for harmonic morphisms as follows.\par

\begin{definition}\label{def7.2}
Let $\varphi=(\phi, \{\phi_v\}_{v\in V(G)})$ be a harmonic morphism of metrized complexes $\mathfrak{C}$, $\mathfrak{C'}.$  For any vertex $v\in V(G)$, the degree of a harmonic morphism $\varphi$ is defined to be equal to the degree of $\phi,$ that is
$$\deg(\varphi):=\deg(\phi)=\sum_{e\in E(G)\atop \phi(e)=e'}U_{\phi}(e)$$ for any $e'\in E(G').$
\end{definition}

Then by Proposition \ref{pro5.3}, for any vertex $v'\in V(G')$ we have the formula \begin{equation}\label{EQ1}\deg(\varphi)=\deg(\phi)=\sum_{v\in V(G)\atop \phi(v)=v'}M_{\phi}(v)=\sum_{v\in V(G)\atop \phi(v)=v'}\deg(\phi_v).\end{equation}

Let $v'$ be a vertex in $G'$ and $x'$ be a point in the associated curve $\{\mathcal{C}_{v'}^{'}\}$ of $v'$ in $\mathfrak{C'}$. Let $E_{x'}$ be the degree one effective divisor on $\mathfrak{C'}$ whose only supporting point is $x'$. We give the definition of pullback divisor according to \cite{LM-2018}.\par

\begin{definition} \label{def7.4} The pullback divisor $\varphi^*(E_{x'})\in Div(\mathfrak{C})$ of $E_{x'}$ is defined as follows:
 \begin{itemize}
   \item the $\Gamma$-part of $\varphi^*(E_{x'})$ is the pullback divisor $\phi^*((v'))\in Div(\Gamma)$ of the divisor $(v')\in Div({\Gamma}')$,
   \item the ${\mathcal{C}}_v$-part of $\varphi^*(E_{x'})$ is the pullback divisor $\phi_{v}^*((x'))\in Div({\mathcal{C}}_v)$ of the divisor $(x')\in Div(\mathcal{C}_{v'}^{'})$ if $v\in {\phi}^{-1}(v')$,
   \item the ${\mathcal{C}}_v$-part of $\varphi^*(E_{x'})$ is 0 if $v\notin {\phi}^{-1}(v')$.
 \end{itemize}
\end{definition}
Note that the properties of harmonic morphisms guarantee that $\varphi^*(E_{x'})$ is a well-defined divisor on $\mathfrak{C}$. We may also simply call the pullback divisor of $E_{x'}$ as the pullback divisor of the point $x'$ sometimes. Moreover, by letting $\varphi^{*}$ preserve linear combinations, we can also naturally associate a pullback divisor $\varphi^{*}(D')$ on $\mathfrak{C}$ to all divisors $D'$ on $\mathfrak{C'}.$\par

For any $v^{'}\in V(G^{'}),$ let $x^{'}\in \mathcal{A}_{v^{'}}.$ Then we can write $$\varphi^*(x')=\sum_{v\in {\phi}^{-1}(v')}\phi_{v}^{*}(x')=\sum_{v\in {\phi}^{-1}(v')}\sum_{x\in\phi_{v}^{-1}(x^{'})\atop v\in e}U_{\phi}(e)(x),$$ where $U_{\phi}(e)$ is equal to the ramification index of $\phi_v$ at the marked point $x$ corresponding to the edge $e\ni v.$
Hence,
$$\deg(\varphi^*(x'))=\sum_{v\in {\phi}^{-1}(v')}\sum_{x\in\phi_{v}^{-1}(x^{'})\atop v\in e}U_{\phi}(e)=\sum_{v\in V(G)\atop \phi(v)=v'}M_{\phi}(v)=\deg(\phi)=\deg(\varphi).$$
Since the $A_{v^{'}}$ is the divisor consisting of the sum of $val(v^{'})$ points in $\mathcal{A}_{v^{'}},$ we get the result.\par

\begin{proposition}\label{P7.4} For the divisor $A_{v^{'}},$ we have $$\deg(\varphi^*(A_{v^{'}}))=\deg(\varphi)\deg(A_{v^{'}}).$$
\end{proposition}\par

\begin{lemma} \label{lem7.5}Let $\varphi:\mathfrak{C}\to \mathfrak{C}'$ be a harmonic morphism. Then for a canonical divisor $\mathcal{K}_{(\mathfrak{C}^{'}, w')},$ we have $$\deg({\varphi}^*(\mathcal{K}_{(\mathfrak{C}^{'}, w')}))=\deg(\varphi)\deg(\mathcal{K}_{(\mathfrak{C}^{'}, w')}).$$ \end{lemma}

\begin{proof}By Definition \ref{D7.1} and Definition \ref{def5.4}, we obtain
\begin{eqnarray*}&&\varphi^{*}(\mathcal{K}_{(\mathfrak{C}^{'}, w')})\\&=&\varphi^{*}\left(\sum_{v^{'} \in V(G^{'})}(K_{v^{'}}+A_{v^{'}}+2w'(v')(v'))\right)\\&=&\sum_{v^{'} \in V(G^{'})}\sum_{v\in V(G)\atop \phi(v)=v^{'}}\phi_{v}^{*}K_{v^{'}}+\sum_{v^{'} \in V(G^{'})}\sum_{v\in V(G)\atop \phi(v)=v^{'}}\phi_{v}^{*}A_{v^{'}}+\sum_{v^{'} \in V(G^{'})}2w'(v')\phi^{*}(v')
\\&=&\sum_{v^{'} \in V(G^{'})}\sum_{v\in V(G)\atop \phi(v)=v^{'}}\phi_{v}^{*}K_{v^{'}}+\sum_{v^{'} \in V(G^{'})}\sum_{v\in V(G)\atop \phi(v)=v^{'}}\phi_{v}^{*}A_{v^{'}}+\sum_{v^{'} \in V(G^{'})}2w'(v')\sum_{v\in V(G)\atop \phi(v)=v^{'}}M_{\phi}(v)(v)\\
&=&\sum_{v^{'} \in V(G^{'})}\sum_{v\in V(G)\atop \phi(v)=v^{'}}\phi_{v}^{*}K_{v^{'}}+\sum_{v^{'} \in V(G^{'})}\sum_{v\in V(G)\atop \phi(v)=v^{'}}\phi_{v}^{*}A_{v^{'}}+\sum_{v^{'} \in V(G^{'})}2w'(v')\deg(\phi)(v).\end{eqnarray*}
Then it follows from \eqref{EQ1} and Proposition \ref{P7.4} that
\begin{eqnarray*}&&\deg(\varphi^{*}(\mathcal{K}_{(\mathfrak{C}^{'}, w')}))\\
&=&\sum_{v^{'} \in V(G^{'})}\sum_{v\in V(G)\atop \phi(v)=v^{'}}\deg(\phi_{v}^{*}K_{v^{'}})+\sum_{v^{'} \in V(G^{'})}\sum_{v\in V(G)\atop \phi(v)=v^{'}}\deg(\phi_{v}^{*}A_{v^{'}})+\deg(\phi)\sum_{v^{'} \in V(G^{'})}2w'(v')\\
&=&\sum_{v^{'} \in V(G^{'})}\sum_{v\in V(G)\atop \phi(v)=v^{'}}\deg(\phi_{v})\deg(K_{v^{'}})+\sum_{v^{'} \in V(G^{'})}\sum_{v\in V(G)\atop \phi(v)=v^{'}}\deg(\phi_{v}^{*}A_{v^{'}})+\deg(\phi)\sum_{v^{'} \in V(G^{'})}2w'(v')\\
&=&\deg(\varphi)\sum_{v^{'} \in V(G^{'})}\deg(K_{v^{'}})+\deg(\varphi)\sum_{v^{'} \in V(G^{'})}\deg(A_{v^{'}})+\deg(\phi)\sum_{v^{'} \in V(G^{'})}2w'(v')\\
&=&\deg(\varphi)\sum_{v^{'} \in V(G^{'})}(2g_{v^{'}}-2)+\deg(\varphi)2|E(G^{'})|+\deg(\phi)\sum_{v^{'} \in V(G^{'})}2w'(v')\\
&=&\deg(\varphi)\left(2\left(|E(G^{'})|-|V(G^{'})|+1\right)+2\sum_{v^{'}\in V(G^{'})}g_{v^{'}}-2+\sum_{v^{'} \in V(G^{'})}2w'(v')\right)\\
&=&\deg(\varphi)\left(2g(\Gamma^{'})+2\sum_{v^{'}\in V(G^{'})}g_{v^{'}}-2+\sum_{v^{'} \in V(G^{'})}2w'(v')\right)\\
&=&\deg(\varphi)\left(2g(\mathfrak{C}^{'})-2+\sum_{v^{'} \in V(G^{'})}2w'(v')\right)\\
&=&\deg(\varphi)\left(2g(\mathfrak{C}^{'}, w')-2\right)\\
&=&\deg(\varphi)\deg(\mathcal{K}_{(\mathfrak{C}^{'}, w')}).\end{eqnarray*}
\end{proof}

\subsection{Riemann-Hurwitz theorem for metric complex of algebraic curves} Now we consider the Riemann-Hurwitz theorem for harmonic morphisms on metric complexes of algebraic curves.\par

\begin{theorem}\label{the7.6} Let $\mathfrak{C}$, $\mathfrak{C'}$ be metrized complexes of algebraic curves on $\kappa$ with genus $g(\mathfrak{C})$ and $g(\mathfrak{C'})$, respectively. The underlying vertex-weighted metric graphs of $\mathfrak{C}$ and $\mathfrak{C'}$ are $(\Gamma, w)$ and $({\Gamma}', w')$ respectively, suppose $(G, \ell)$ and $(G', {\ell}')$ are loopless models for metric graphs $\Gamma$ and ${\Gamma}'$ respectively, and $\varphi=(\phi, \{\phi_v\}_{v\in G})$ be a harmonic morphism. Then, the canonical divisors on  $\mathfrak{C}$ and  $\mathfrak{C'}$ are related by the formula
$$\mathcal{K}_{(\mathfrak{C}, w)}={\varphi}^*\mathcal{K}_{(\mathfrak{C'}, w')}+R_{\varphi},$$
where
\begin{eqnarray*}
R_{\varphi}&:=&\sum_{v\in V(G)}\left(K_v+A_v-\varphi^{*}K_{v'}-\varphi^{*}A_{v^{'}}\right)+\sum_{v\in V(G)}2\left(w(v)-w'(v')M_{\phi}(v)\right)(v),
\end{eqnarray*} and $v'=\phi(v).$ In addition,
\begin{eqnarray*}
2g(\mathfrak{C})-2 &=& \deg(\varphi)(2g(\mathfrak{C'})-2)+\sum_{v\in V(G)}2\left(M_{\phi}(v)-1+g_v-M_{\phi}(v)g_{v'}\right)\\&&+\sum_{v\in V(G)}\left(val(v)-M_{\phi}(v)val(v')\right)+\sum_{v\in V(G)}2\left(w(v)-w'(v')M_{\phi}(v)\right).
\end{eqnarray*}
\end{theorem}

\begin{remark}If the underlying metric graphs $\Gamma$ and ${\Gamma}'$  have no weights at all vertexes in Theorem \ref{the7.6}, then it
reduces that $$\mathcal{K}_{\mathfrak{C}}={\varphi}^*\mathcal{K}_{\mathfrak{C'}}+R_{\varphi},$$
where $R_{\varphi}=\sum_{v\in V(G)}\left(K_v+A_v-\varphi^{*}K_{v'}-\varphi^{*}A_{v^{'}}\right).$  In addition,
\begin{eqnarray*}
2g(\mathfrak{C})-2 &=& \deg(\varphi)(2g(\mathfrak{C'})-2)+\sum_{v\in V(G)}2\left(M_{\phi}(v)-1+g_v-M_{\phi}(v)g_{v'}\right)\\&&+\sum_{v\in V(G)}\left(val(v)-M_{\phi}(v)val(v')\right).
\end{eqnarray*}
\end{remark}

\begin{proof}
For every $v\in V(G)$ we have $${\mathcal{K}}_{(\mathfrak{C}, w)}(v)=K_v+A_v+2w(v)(v),$$ where $K_v$ is a canonical divisor on $\mathcal{C}_v,$  $A_v$ is the divisor consisting of the sum of the $val(v)$ points in $\mathcal{A}_v, $ and $w(v)$ is the weight at vertex $v\in V(G).$ Let $v^{'}=\phi(v).$ Then for each $v\in V(G),$ we have \begin{eqnarray*}&& {\mathcal{K}}_{(\mathfrak{C}, w)}(v)-\varphi^{*}(\mathcal{K}_{(\mathfrak{C^{'}}, w')})(v)\\ &=&\left(K_v+A_v+2w(v)(v)\right)-\left(\varphi^{*}K_{v'}+\varphi^{*}A_{v^{'}}+2w'(v')\phi^{*}(v')\right)
\\\nonumber&=&\left(K_v+A_v+2w(v)(v)\right)-\left(\varphi^{*}K_{v'}+\varphi^{*}A_{v^{'}}+2w'(v')M_{\phi}(v)(v)\right)\\&=&R_{\varphi}.\end{eqnarray*} Hence , we obtain $\mathcal{K}_{(\mathfrak{C}, w)}={\varphi}^*\mathcal{K}_{(\mathfrak{C'}, w')}+R_{\varphi}.$\par

Now by Lemma \ref{lem7.5},  it follows that
\begin{eqnarray*}&&
2g(\mathfrak{C}, w)-2 - \deg(\varphi)(2g(\mathfrak{C'}, w')-2)\\
&=&\deg(\mathcal{K}_{(\mathfrak{C}, w)})-\deg(\varphi)\deg(\mathcal{K}_{(\mathfrak{C}^{'}, w')})\\
&=&\deg(\mathcal{K}_{(\mathfrak{C}, w)})-\deg({\varphi}^*\mathcal{K}_{(\mathfrak{C'}, w')})\\
&=&\deg(R_{\varphi})\\
&=&\sum_{v\in V(G)}\left(\deg(K_v)+\deg(A_v)+2w(v)\right)\\&&-\sum_{v\in V(G)}\left(\deg(\phi_{v}^{*}K_{v'})+\deg(\phi_{v}^{*}A_{v^{'}})+2w'(v')M_{\phi}(v)\right)\\
&=&\sum_{v\in V(G)}\left(\deg(K_v)+\deg(A_v)+2w(v)\right)\\&&-\sum_{v\in V(G)}\left(\deg(\phi_{v})\deg(K_{v'})+\deg(\phi_{v})\deg(A_{v^{'}})+2w'(v')M_{\phi}(v)\right)\\
&=&\sum_{v\in V(G)}\left(2g_v-2+val(v)+2w(v)\right)-\sum_{v\in V(G)}M_{\phi}(v)\left(2g_{v^{'}}-2+val(v^{'})\right)\\&&-\sum_{v\in V(G)}2w'(v')M_{\phi}(v)\nonumber\\
&=& \sum_{v\in V(G)}2\left(M_{\phi}(v)-1+g_v-M_{\phi}(v)g_{v'}+w(v)-w'(v')M_{\phi}(v)\right)\\&&+\sum_{v\in V(G)}\left(val(v)-M_{\phi}(v)val(v')\right).
\end{eqnarray*} We complete the proof of this theorem.
\end{proof}

\subsection{Second main theorem for metric complexes of algebraic curves}
In the final subsection, we obtain the second main theorem of harmonic morphisms on metric complexes of algebraic curves.\par

\begin{theorem}\label{T7.7} Let $\mathfrak{C}$, $\mathfrak{C'}$ be metrized complexes of algebraic curves over $\kappa,$ the underlying vertex-weighted metric graphs of $\mathfrak{C}$ and $\mathfrak{C'}$ are $(\Gamma, w)$ and $({\Gamma}', w')$ respectively, suppose $(G, \ell)$ and $(G', {\ell}')$ are loopless models for metric graphs $\Gamma$ and ${\Gamma}'$ respectively, and $\varphi=(\phi, \{\phi_v\}_{v\in G})$ be a harmonic morphism. Suppose that $a_1, \ldots , a_q\in V(G')$ are distinct vertices. Let $E={\phi}^{-1}(\{a_1, \ldots, a_q\}).$ Then we have
\begin{eqnarray*}
(q+g(\mathfrak{C'}, w')-1)\deg(\varphi)&\leq& g(\mathfrak{C}, w)-1+|E\cap V(G)|-\sum_{v\in V(G)}\left(g_v-M_{\phi}(v)g_{v^{'}}\right)\nonumber\\&&-\sum_{v\in V(G)}\left(w(v)-w'(v')M_{\phi}(v)\right)\\
&\;&-\frac{1}{2}\sum_{v\in V(G)}(val(v)-M_{\phi}(v)val(v')),
\end{eqnarray*}
where $v'=\phi(v)$ and $|E\cap V(G)|$ is the cardinality of $E\cap V(G).$
\end{theorem}

\begin{proof}
Set $$r_{\varphi}(E):=\sum_{v\in E\cap V(G)}(M_{\phi}(v)-1),$$ and $$r_{\varphi}(\mathfrak{C}):=\sum_{v\in V(G)}(M_{\phi}(v)-1).$$ It is obvious that $$r_{\varphi}(E)\leq r_{\varphi}(\mathfrak{C}).$$ From the definition of the degree of a harmonic morphism $\varphi: \mathfrak{C}\to \mathfrak{C'}$, we get that $$\sum_{v\in E\cap V(G)\atop a_j=\phi(v)}M_{\phi}(v)=\sum_{v\in V(G)\atop a_j=\phi(v)}M_{\phi}(v)=\deg(\phi)=\deg(\varphi)$$ holds for each $ a_j\in \{a_1, \ldots, a_q\}.$ Then we have  \begin{eqnarray*}r_{\varphi}(\mathfrak{C})\geq r_{\varphi}(E)&=&\sum_{v\in E\cap V(G)}(M_{\phi}(v)-1)\\\nonumber&=&(\sum\limits_{j=1}^q \sum\limits_{v\in E\cap V(G)\atop a_j=\phi(v)}M_{\phi}(v))-|E\cap V(G)|\\\nonumber&=&q\deg(\varphi)-|E\cap V(G)|.\end{eqnarray*}  On the other hand, by the Riemann-Hurwitz theorem (Theorem \ref{the7.6}), we have
\begin{eqnarray*}r_{\varphi}(\mathfrak{C})&=&
\sum_{v\in V(G)}(M_{\phi}(v)-1)\\&=& (g(\mathfrak{C}, w)-1)- \deg(\varphi)(g(\mathfrak{C'}, w')-1)-\sum_{v\in V(G)}(g_v-M_{\phi}(v)g_{v^{'}})\nonumber\\
&\;&-\sum_{v\in V(G)}\left(w(v)-w'(v')M_{\phi}(v)\right)-\frac{1}{2}\sum_{v\in V(G)}\left(val(v)-M_{\phi}(v)val(v')\right).
\end{eqnarray*}
Hence, we get the following inequality
\begin{eqnarray*}
&&(q+g(\mathfrak{C'}, w')-1)\deg(\varphi) \leq g(\mathfrak{C}, w)-1+|E\cap V(G)|-\sum_{v\in V(G)}(g_v-M_{\phi}(v)g_{v^{'}})\nonumber\\
&&-\sum_{v\in V(G)}\left(w(v)-w'(v')M_{\phi}(v)\right)-\frac{1}{2}\sum_{v\in V(G)}(val(v)-M_{\phi}(v)val(v')).
\end{eqnarray*}
\end{proof}

If the underlying metric graphs have no weights at all vertexes, then Theorem \ref{T7.7} yields the following corollary.\par

\begin{corollary} Let $\mathfrak{C}$, $\mathfrak{C'}$ be metrized complexes of algebraic curves over $\kappa,$  the underlying metric graphs of $\mathfrak{C}$ and $\mathfrak{C'}$ are $\Gamma$ and ${\Gamma}'$ respectively, suppose $(G, \ell)$ and $(G', {\ell}')$ are loopless models for metric graphs $\Gamma$ and ${\Gamma}'$ respectively, and $\varphi=(\phi, \{\phi_v\}_{v\in G})$ be a harmonic morphism. Suppose that $a_1, \ldots , a_q\in V(G')$ are distinct vertices. Let $E={\phi}^{-1}(\{a_1, \ldots, a_q\}).$ Then we have
\begin{eqnarray*}
(q+g(\mathfrak{C'})-1)\deg(\varphi)&\leq& g(\mathfrak{C})-1+|E\cap V(G)|-\sum_{v\in V(G)}\left(g_v-M_{\phi}(v)g_{v^{'}}\right)\nonumber\\&&
-\frac{1}{2}\sum_{v\in V(G)}(val(v)-M_{\phi}(v)val(v')),
\end{eqnarray*}
where $v'=\phi(v)$ and $|E\cap V(G)|$ is the cardinality of $E\cap V(G).$
\end{corollary}

At the end of this section, we give an example to show the second main theorem for harmonic morphism on metrized complexes of algebraic curves without weights at vertexes.\par

\begin{figure}
  \centering
  \includegraphics[width=8cm]{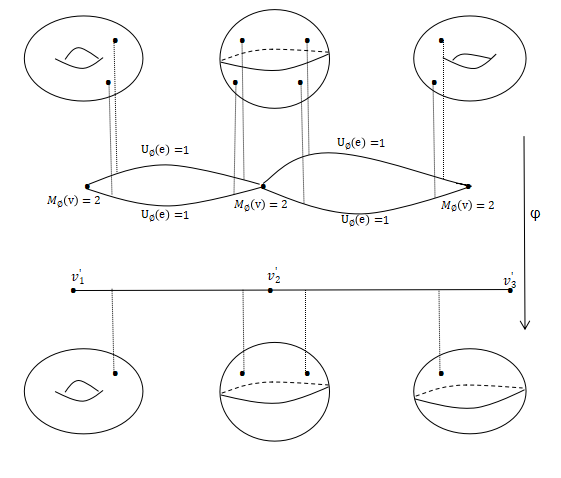}\\
  \caption{A harmonic morphism of degree two}\label{fig.8}
\end{figure}

\begin{example} Let $\mathfrak{C}$, $\mathfrak{C'}$ be metrized complexes of algebraic curves which genus $g(\mathfrak{C})=4$ and $g(\mathfrak{C'})=1$, respectively.  We give a harmonic morphism $\varphi$ (for suitable choices of lengths)  depicted in Figure \ref{fig.8}. We do not specify the lengths of edges of $\Gamma$ and ${\Gamma}'$.  It is known by calculation that $\deg(\varphi)=2,$ $\frac{1}{2}\left(val(v)-M_{\phi}(v)val(v')\right)=0.$\par

\begin{itemize}
  \item Take $q=3$ and let $\{a_1, \ldots, a_q\}=\{v_{1}', \ldots, v_{3}'\}=V(G^{'}).$ Then one can get that $|E\cap V(G)|=3$ and $\sum\limits_{v\in V(G)}(g_v-M_{\phi}(v)g_{v^{'}})=0.$ Then the second main theorem (Theorem \ref{T7.7}) gives $(3+1-1)\times 2\leq 4-1+3=6.$  This means that the inequality of the second main theorem is sharp.
  \item  Take $q=2.$ \begin{itemize}
                       \item Let $\{a_1, a_q\}=\{v_{1}',v_{2}'\}\subset V(G')$. Then $|E\cap V(G)|=2$ and $\sum\limits_{v\in V(G)}(g_v-M_{\phi}(v)g_{v^{'}})=-1,$ so we have the inequality $4 \leq 4-1+2+1=6$ by the second main theorem.
                       \item Let $\{a_1, a_q\}=\{v_{2}',v_{3}'\}\subset V(G')$. Then $|E\cap V(G)|=2$ and $\sum\limits_{v\in V(G)}(g_v-M_{\phi}(v)g_{v^{'}})=1,$  so we have the equality $4 \leq 4-1+2-1=4$ by the second main theorem.
                       \item Let $\{a_1, a_q\}=\{v_{1}',v_{3}'\}\subset V(G')$. Then $|E\cap V(G)|=2$ and $\sum\limits_{v\in V(G)}(g_v-M_{\phi}(v)g_{v^{'}})=0,$  so we have the inequality $4 \leq 4-1+2=5$ by the second main theorem.
                     \end{itemize}
  \item Take $q=1.$ \begin{itemize}
                      \item Let $a_q\in V(G')$ take $v_{1}'$. Then $|E\cap V(G)|=1$ and $\sum\limits_{v\in V(G)}(g_v-M_{\phi}(v)g_{v^{'}})=-1,$ so we have the inequality $2 \leq 4-1+1+1=5$ by the second main theorem.
                      \item Let $a_q\in V(G')$ take $v_{2}'$. Then $|E\cap V(G)|=1$, and$\sum\limits_{v\in V(G)}(g_v-M_{\phi}(v)g_{v^{'}})=0,$ so we have the inequality $2 \leq 4-1+1=4$ by the second main theorem.
                      \item Let $a_q\in V(G')$ take $v_{3}'$. Then $|E\cap V(G)|=1$, and$\sum\limits_{v\in V(G)}(g_v-M_{\phi}(v)g_{v^{'}})=1,$ so we have the inequality $2 \leq 4-1+1-1=3$ by the second main theorem.
                    \end{itemize}
\end{itemize}
\end{example}

\end{document}